\documentclass[11pt,reqno]{amsart}
\usepackage[utf8]{inputenc}
\usepackage[T1]{fontenc}
\usepackage{lmodern}
\usepackage[english]{babel}
\usepackage{amsmath,a4wide}
\usepackage{xfrac}
\usepackage{esint}
\usepackage{mathrsfs,bm,amsthm,mathtools,yfonts,amssymb,color,braket,booktabs,graphicx,graphics,amsfonts,comment,siunitx}

\newcommand{\R}{\mathbb{R}}

\newcommand{\N}{\mathbb{N}}
\newcommand{\V}{\mathbb{V}}
\newcommand{\RV}{\mathbb{RV}}
\newcommand{\IV}{\mathbb{IV}}
\newcommand{\D}{\nabla}
\newcommand{\pt}{\partial_t}
\newcommand{\dv}{\mathrm{div}}
\newcommand{\supp}{\mathrm{supp}\,}
\newcommand{\e}{\mathbf{e}}
\newcommand{\lv}{\lvert}
\newcommand{\rv}{\rvert}
\newcommand{\lV}{\lVert}
\newcommand{\rV}{\rVert}

\theoremstyle{plain}
\newtheorem{theorem}{Theorem}[section]
\newtheorem*{theorem*}{Theorem}
\newtheorem{lemma}[theorem]{Lemma}
\newtheorem*{lemma*}{Lemma}
\newtheorem{prop}[theorem]{Proposition}
\newtheorem*{prop*}{Proposition}
\newtheorem{corollary}[theorem]{Corollary}
\newtheorem*{corollary*}{Corollary}
\theoremstyle{definition}

\newtheorem*{example*}{e.g.}
\newtheorem{remark}[theorem]{Remark}
\newtheorem*{remark*}{Remark}

\newtheorem*{assumption*}{Assumption}
\numberwithin{equation}{section}
\newtheorem{definition}[theorem]{Definition}
\newtheorem*{definition*}{Definition}

\title{Existence of weak mean curvature flow
	with prescribed contact angle via elliptic regularization}
\author{Kiichi Tashiro}
\address{(K.Tashiro) Dipartimento di Matematica, Universit\`{a} degli Studi di Milano, Via Saldini 50, I-20133 Milano (MI), Italy (former: Department of Mathematics, Institute of Science Tokyo, 2-12-1 Ookayama, Meguroku, Tokyo 152-8551, Japan)}
\email{kiichi.tashiro@unimi.it (former: tashiro.k.ai@m.titech.ac.jp, tashiro.k.0e2f@m.isct.ac.jp)}

\begin{document}
	
	\begin{abstract}
		In the present paper, we study the existence of Brakke-type weak mean curvature flow satisfying a prescribed contact angle condition for a general angle $ \theta \in ( 0 , \pi ) $ via Ilmanen's regularization. The main ingredients of the result are the extension of Ilmanen's regularization to the capillarity and the derivation of the first variation estimates for the interior and wetted boundary varifolds separately.
	\end{abstract}
	
	\maketitle
	\markboth{Kiichi Tashiro}{Weak mean curvature flow with contact angle}
	
	\section{Introduction}
	The mean curvature flow (henceafter referred to as MCF), which emerged in the context of material science,
	has been studied by numerous researchers as one of the most important geometric flow problems nowadays. The unknown of the MCF is a one-parameter family $ \{ \Gamma_t \}_{ t \geq 0 } $ such that the normal velocity of $ \Gamma_t $ equals its mean curvature vector at each point for every time. As a typical boundary condition, we study the contact angle condition, which is dedicated to surface tension and the wettability of the container, for the MCF. In general, it is well known that singular behaviors such as shrinking and neck pinching occur in MCFs. Moreover, a MCF with the contact angle may pop upon tangential contact with the boundary. To consider the solutions that allow such singularities, we need to use extensive tools from geometric measure theory and a weak notion of MCFs.
	\vskip.3\baselineskip
	The purpose of this paper is to investigate capillary boundary conditions for geometric flow problems. As a foundational result in the varifold setting, Kagaya and Tonegawa \cite{kagaya2017contactangle} introduced a notion of contact angle condition and established a monotonicity formula extending the free-boundary monotonicity formula of Gr\"{u}ter and Jost \cite{gruter1986allard} to the contact angle setting. Subsequently, De Masi investigated the rectifiability of the contact set between a varifold and the boundary in \cite{de2021rectifiability}, while De Masi and De Philippis developed a min-max theory for varifolds with contact angle boundary conditions in \cite{de2025min}. These results are summarized in De Masi's Ph.D. thesis \cite{de2022existence}. There have been significant progress on the regularity and geometric properties of contact angle varifolds and capillary problems arising from Gauss' free energy functional and we briefly mention recent studies; \cite{bevilacqua2025classical,de2025regularity,de2024regularity,kagaya2018singular,king2022smoothness,king2022plateau,wang2024allard,wang2024monotonicity}. On the other hand, there are fewer studies on the MCF with contact angle. In this paper, we consider a notion of Brakke-type mean curvature flows in the framework of geometric measure theory and establish their global-in-time existence via the elliptic regularization \cite{ilmanen1994elliptic} under the general situation whenever possible. Roughly speaking, the main result is the following:
	\begin{theorem*}
		Let $ M $ be a closed bounded domain with the smooth boundary and $ E_0 \subset M $ be a set of finite perimeter. Then there exists a Brakke flow $ \{ V_t \}_{ t \geq 0 } $ in $ M $ such that $ \lV V_0 \rV = \mathcal{ H }^n \llcorner_{ \partial^* E_0 \cap M^{ \circ } } $, $ V_t $ has the contact angle $ \theta $ and is integral for almost every time $ t \geq 0 $. That is to say, for all $ \phi \in C^1_c ( M \times [ 0 , \infty ) ; [ 0 , \infty ) ) $ and all $ 0 \leq t_1 < t_2 < \infty $, we have
		\[
		\lV V_{ t_2 } \rV ( \phi ( \cdot , t_2 ) ) - \lV V_{ t_1 } \rV ( \phi ( \cdot , t_1 ) ) \leq \int^{ t_2 }_{ t_1 } \int_M ( \D \phi - \phi H_{ V_t } ) \cdot H_{ V_t } + \pt \phi\,d \lV V_t \rV d t,
		\]
		where $ H_{ V_t } $ is the generalized mean curvature of $ V_t $.
	\end{theorem*}
	\begin{figure}[h]
		\centering
		\includegraphics[width=0.75\linewidth]{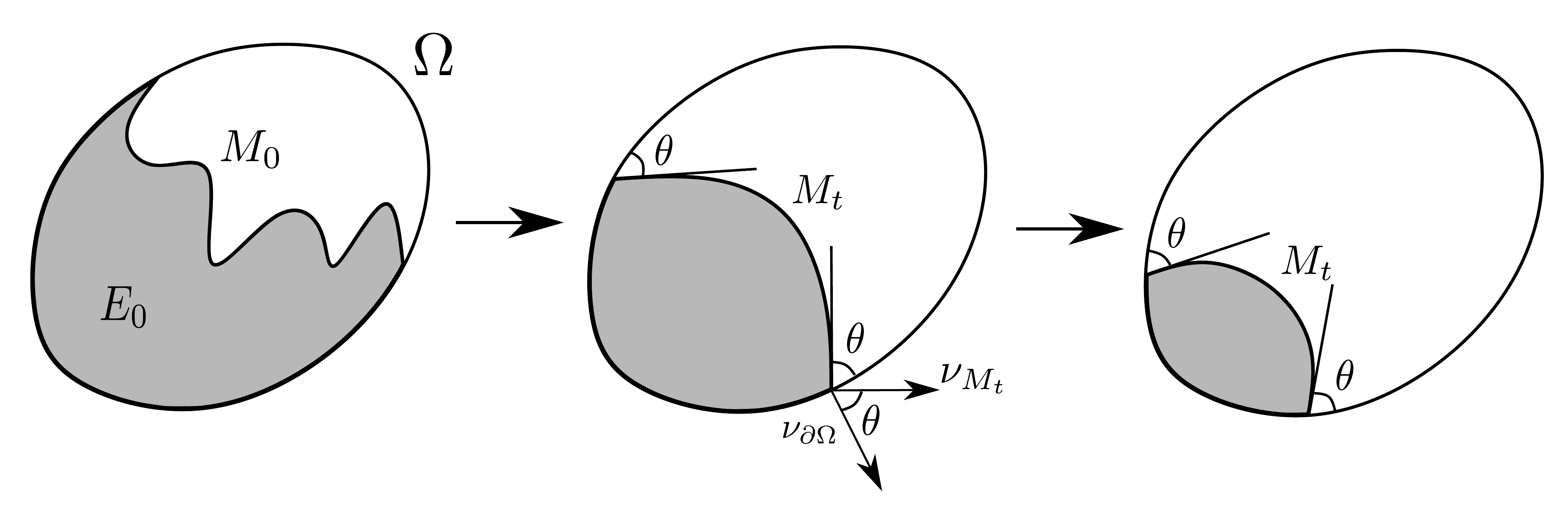}
		\caption{}
		\label{contact_angle_MCF}
	\end{figure}
	Figure \ref{contact_angle_MCF} provides a visualization of this flow. The exact statement will be given in Section \ref{sec_mainresult}. Brakke \cite{brakke1978motion} proposed and studied Brakke's MCF by characterizing the motion law of the surfaces using the above inequality. Ilmanen \cite{ilmanen1994elliptic} studied the existence and various properties of
	(unconstrained) Brakke flows using the elliptic regularization method. One advantage of this method is that one can apply White's local regularity theorem \cite{brian2005regularity} for this Brakke flow. It is not known, but White's study provides a basis for expecting that one can establish the regularity theorem for Brakke flows with the contact angle. We mention known studies on MCFs with contact angle: Hensel and Laux \cite{hensel2021existence} proved the existence and the weak-strong uniqueness of BV flow, which is based on the framework of BV functions originally proposed by Luckhaus and Sturzenhecker \cite{luckhaus1995implicit}, with the contact angle via the Allen--Cahn equation under a reasonable assumption. Marshall-Stevens et al. \cite{marshallstevens2024gradientflowphasetransitions} studied the gradient flow of the Allen--Cahn equation and convergence to a Brakke flow under general assumptions. The theory of viscosity solutions for capillary MCFs has been developed in \cite{barles1999nonlinear,ishii2004nonlinear}. More recently, the existence of weak solutions constructed via minimizing movement-type schemes, as well as their relation to viscosity solutions, has been actively investigated. Eto and Giga \cite{eto2024minimizing,eto2024convergence} studied a Chambolle-type minimizing movement scheme with capillarity and discussed the existence of minimizers and their convergence to viscosity solutions under a contact angle boundary condition. Bellettini and Kholmatov studied the capillary Almgren--Taylor--Wang scheme for MCFs of droplets with a contact angle condition and established various comparison results in \cite{bellettini2018minimizing}. As a continuation of this line of research, Kholmatov proved its consistency with smooth MCFs in \cite{kholmatov2024consistency} and investigated minimizing movements for forced anisotropic MCFs of droplets in \cite{kholmatov2024minimizing}.
	In contrast, we prove the existence of a weak solution of MCFs based on Brakke's work without any extra assumption in the co-dimension $ 1 $ case.
	\vskip.3\baselineskip
	We also briefly mention closely related works on MCFs using the elliptic regularization. Edelen \cite{Edelen+2020+95+137} proved the existence of Brakke flows with the Neumann boundary condition and the regularity theorem for it, and White \cite{WhiteMCF2021} studied the Dirichlet boundary condition for Brakke flows via the elliptic regularization. Schulze and White \cite{SchulzeWhite+2020+281+305} utilized the elliptic regularization to construct a MCF with a triple junction by working with the class of flat chains.
	\vskip.3\baselineskip
	The key elements of the present paper are the establishment of the extension of Ilmanen's regularization to the capillary setting and the derivation of first variation estimates for interior and wetted boundary surfaces. In the framework of contact angle varifolds, boundary contact at a prescribed angle $ \theta $ is represented by a pair of varifolds: one on the boundary and the other within the domain. This is the main difference from the standard elliptic regularization \cite{Edelen+2020+95+137,ilmanen1994elliptic}. A challenge arises from the fact that the estimates of the first variations of contact angle varifolds cannot be derived from the definition. This difficulty originates from the inability to control tangential directional variations along the boundary from the definition. To address this problem, we estimate them through the energy-minimizing structure inherent in the elliptic regularization. De Philippis and Maggi \cite{philippis2015regularity} proved the regularity property for the (almost) minimizer with capillarity. Using their results and De Masi's estimate \cite[Theorem 1.1]{de2021rectifiability}, we prove that the first variations are uniformly locally finite based solely on normal directional variations. This argument allows us to prove the compactness theorem for the approximated Brakke flows from the elliptic regularization with capillarity, thereby establishing the existence of the flow.
	\vskip.3\baselineskip
	The paper is organized as follows. In Section \ref{Basic_notation}, we set our notation and explain the definition of the Brakke-type MCF and the main result. In Section \ref{cpt_contactangle_varifolds}, we explain the compactness theorem for contact angle varifolds under a strong assumption. In Section \ref{cpt_contactangle_Brakke}, we prove the compactness theorem for contact angle Brakke flows by utilizing the results in Section \ref{cpt_contactangle_varifolds}. In Section \ref{elliptic_regularization}, we adopt the contact angle condition to the elliptic regularization and prove the $ \varepsilon $-independent estimates of the first variation. As a consequence, we finally prove the main result of this paper.
	
	\section*{Acknowledgement}
	The author would like to express his deepest gratitude to his supervisor Yoshihiro Tonegawa and his friend Kotaro Motegi for their insightful feedback to improve the quality of this paper. He would also like to thank Dr. Xuwen Zhang for his pointing out the first draft. This work was supported by JST SPRING, Japan Grant Number JPMJSP2106 and 2180, and the Grand-in-Aid for JSPS Fellows, Grant number 25KJ1245.
	
	\section{Preliminaries and main results}
	\label{Basic_notation}
	\subsection{Basic notation}
	The ambient space in which we will work is the Euclidean space $ \R^{ n + 1 } $. For each $ A \subset \R^{ n + 1 } $, $ \chi_A $ denotes the indicator function of $ A $, $ \overline{ A } $ and $ A^{ \circ } $ the closure and the interior of $ A $ in the Euclidean topology, respectively. When $ x \in \R^{ n + 1 } $ and $ r > 0 $, $ B_r ( x ) $ denotes the open ball with centre $ x $ and radius $ r $. For any integer $ k > 0 $, the symbols $ \mathcal{ L }^k $ and $ \mathcal{ H }^k $ denote the ($ k $-dimensional) Lebesgue measure and Hausdorff measure, respectively. For each integer $ k \geq 1 $, let $ \omega_k $ denote the volume of the unit ball in $ k $-dimensional space. The symbols $ \D $, $ \D' $, $ \Delta $, $ \D^2 $ denote the spatial gradient and the full gradient in $ \R^{ n + 1 } \times \R $, Laplacian and Hessian, respectively. As a class of test vector fields, we define
	\[
	\mathcal{ T }_{ \Gamma } C^k_c ( U ; \R^{ n + 1 } ) := \{ X \in C^k_c ( U ; \R^{ n + 1 } ) : X ( x ) \in T_x \Gamma \text{ for all } x \in \Gamma \}
	\]
	for an open set $ U \subset \R^{ n + 1 } $ and an $ n $-dimensional $ C^3 $ hypersurface $ \Gamma \subset U $ without boundary.
	\vskip.3\baselineskip
	A positive Radon measure $ \mu $ on $ \R^{ n + 1 } $ (or ``space-time'' $ \R^{ n + 1 } \times [ 0 , \infty ) $) is always regarded as a positive linear functional on the space $ C^0_c ( \R^{ n + 1 } ) $ of continuous and compactly supported functions, with the pairing denoted by $ \mu ( \phi ) $ for $ \phi \in C^0_c ( \R^{ n + 1 } ) $. The restriction of $ \mu $ to a Borel set $ A $ is denoted $ \mu \llcorner_A $, so that $ ( \mu \llcorner_A ) ( E ) := \mu ( A \cap E ) $ for any Borel set $ E \subset \R^{ n + 1 } $. The support of $ \mu $ is denoted $ \supp \mu $, and it is the closed set defined by
	\[
	\supp \mu := \bigl\{ x \in \R^{ n + 1 } : \mu ( B_r ( x ) ) > 0 \text{ for all } r > 0 \bigr\}.
	\]
	For $ 1 \leq p \leq \infty $, the space of $ p $-integrable functions with respect to $ \mu $ is denoted $ L^p ( \mu ) $. For a signed or vector-valued measure $ \mu $, $ \lv \mu \rv $ denotes its total variation. For two Radon measures $ \mu $ and $ \overline{ \mu } $ on $ \R^{ n + 1 } $, when the measure $ \overline{ \mu } $ is absolutely continuous with respect to $ \mu $, we write $ \overline{ \mu } \ll \mu $. When $ \mu $ and $ \overline{ \mu } $ are positive and $ \overline{ \mu } ( \phi ) \leq \mu ( \phi ) $ holds for all $ \phi \in C^0_c ( \R^{ n + 1 } ; [ 0 , \infty ) ) $, we will write $ \overline{ \mu } \leq \mu $.
	\vskip.3\baselineskip
	Let $ U \subset \R^{ n + 1 } $ be an open set. We say that a set $ E \subset U $ (or $ U \times [ 0 , \infty ) $) is a set of locally finite perimeter if, for all bounded open set $ V \subset \subset U $, the set $ E $ satisfies
	\[
	\sup \Biggl\{ \int_{ E \cap V } \dv X\,d x : X \in C^1_c ( V ; \R^{ n + 1 } ) , \lV X \rV_{ C^0 } \leq 1 \Biggr\} < \infty.
	\]
	When the above quantity is finite for $ V = U $, we simply say that $ E $ is a set of finite perimeter. If $ E \subset \R^{ n + 1 } $ is a set of locally finite perimeter, then there exists a vector-valued Radon measure satisfying
	\[
	\int_E \dv X\,d x = - \int_{ \R^{ n + 1 } } X \cdot d \D \chi_E \text{ for all } X \in C^1_c ( \R^{ n + 1 } ; \R^{ n + 1 } ).
	\]
	The derivative measure $ \D \chi_E $ is the associated Gauss--Green measure, and its total variation $ \lV \D \chi_E \rV $ is the perimeter measure; by De Giorgi's structure theorem, $ \lV \D \chi_E \rV = \mathcal{ H }^n \llcorner_{ \partial^* E } $, where $ \partial^* E $ is the reduced boundary of $ E $, and $ \D \chi_E = - \nu_E\,\lV \D \chi_E \rV = - \nu_E\,\mathcal{ H }^n \llcorner_{ \partial^* E } $, where $ \nu_E $ is the outer pointing unit normal vector field to $ \partial^* E $. It is often noted in this paper that $ \nu_E = \nu_F $ on $ \partial^* E \cap \partial^* F $ when $ E \subset F $ are sets of locally finite perimeter (see \cite[Section 16]{maggi2012sets} for example).
	\vskip.03\baselineskip
	A subset $ \Gamma \subset \R^{ n + 1 } $ is countably $ k $-rectifiable if it admits a covering $ \Gamma \subset Z \cup \bigcup_{ i \in \N } f^i ( \R^k ) $, where $ \mathcal{ H }^k ( Z ) = 0 $ and $ f^i : \R^k \to \R^{ n + 1 } $ is Lipschitz. If $ \Gamma $ is countably $ k $-rectifiable, $ \mathcal{ H }^k $-measurable and $ \mathcal{ H }^k ( \Gamma \cap K ) < \infty $ for any compact set $ K \subset \R^{ n + 1 } $, $ \Gamma $ has a measure-theoretic tangent plane called approximate tangent plane for $ \mathcal{ H }^k $-almost every $ x \in \Gamma $ (\cite[Theorem 11.6]{simon1983lectures}), denoted by $ T_x \Gamma $. We may simply refer to it as the tangent plane at $ x \in \Gamma $ without fear of confusion. A Radon measure $ \mu $ is said to be $ k $-rectifiable if there are a countably $ k $-rectifiable, $ \mathcal{ H }^k $-measurable set $ \Gamma $ and a positive function $ \Theta \in L^1_{ loc } ( \mathcal{ H }^k \llcorner_{ \Gamma } ) $ such that $ \mu = \Theta\,\mathcal{ H }^k \llcorner_{ \Gamma } $. This function $ \Theta $ is called multiplicity of $ \mu $. The approximate tangent plane of $ \Gamma $ in this case (which exists $ \mu $-almost everywhere) is denoted by $ T_x \mu $. When $ \Theta $ is an integer for $ \mu $-almost everywhere, $ \mu $ is said to be integral. We say $ \mu $ is a unit density $ k $-rectifiable Radon measure if $ \mu $ is integral and $ \Theta = 1 $ for almost everywhere on $ \Gamma $. 
	\vskip.3\baselineskip
	When $ 1 \leq k \leq n + 1 $, we call $ G ( n + 1 , k ) $ the Grassmannian of the un-oriented $ k $-dimensional linear subspaces of $ \R^{ n + 1 } $. For any open set $ U \subset \R^{ n + 1 } $, let $ G_k ( U ) := U \times G ( n + 1 , k ) $ be the trivial Grassmanian bundle over $ U $. A $ k $-varifold on $ U $ is a positive Radon measure on $ G_k ( U ) $. The set of $ k $-varifolds on $ U $ is denoted by $ \V_k ( U ) $. If the support of $ \lV V \rV $ is contained in a closed set $ M \subset \R^{ n + 1 } $, we may simply denote this by $ V \in \V_k ( M ) $. For a $ k $-varifold $ V $, the mass measure of $ V $ is denoted by $ \lV V \rV $, that is,
	\[
	\lV V \rV ( \phi ) := \int_{ G_k ( U ) } \phi ( x )\, d V ( x , S ) \text{ for all } \phi \in C^0_c ( U ).
	\]
	We say a $ k $-varifold $ V $ is rectifiable if there exists a corresponding $ k $-rectifiable Radon measure $ \mu = \Theta\,\mathcal{ H }^k \llcorner_{ \Gamma } $ such that $ V $ is represented as $ V = \Theta\,\mathcal{ H }^k \llcorner_{ \Gamma } \otimes \delta_{ T_x \Gamma } $, and $ V $ is integral if a corresponding $ k $-rectifiable Radon measure is integral. The set of rectifiable (or integral) $ k $-varifolds on $ U $ is denoted by $ \RV_k ( U ) $ (or $ \IV_k ( U ) $). When $ V $ is integral and a unit density, we say $ V $ is a unit density $ k $-varifold. For any subset $ \mathbf{ F } \subset C^1_c ( U ; \R^{ n + 1 } ) $, the first variation with respect to $ \mathbf{ F } $ of $ V \in \V_k ( U ) $ is defined by
	\[
	\delta V ( X ) := \int_U \dv_S X ( x )\,d V ( x , S ) \text{ for all } X \in \mathbf{ F },
	\]
	where $ \dv_S X ( x ) = \mathrm{ tr } ( S ( \D X ( x ) ) ) $. We say a $ k $-varifold $ V $ has bounded first variation with respect to $ \mathbf{ F } $ if it satisfies $ \sup \{ \lv \delta V ( X ) \rv : X \in \mathbf{ F } , \lV X \rV_{ C^0 } \leq 1 \} < \infty $. If $ V $ has bounded first variation with respect to $ C^1_c ( U ; \R^{ n + 1 } ) $, then by the Lebesgue decomposition theorem, there exist a positive Radon measure $ \lV \delta^s V \rV $ on $ U $, a $ \lV \delta^s V \rV $-measurable vector field $ \eta_V : U \to \R^{ n + 1 } $ and a $ \lV V \rV $-measurable vector field $ H_V : U \to \R^{ n + 1 } $ such that, for all $ X \in C^1_c ( U ; \R^{ n + 1 } ) $,
	\[
	\delta V ( X ) = - \int_U H_V \cdot X\,d \lV V \rV + \int_U X \cdot \eta_V\,d \lV \delta^s V \rV,
	\]
	where $ \lV \delta^s V \rV $ is the singular part of $ \lV \delta V \rV $ with respect to $ \lV V \rV $ which satisfies
	\[
	\lV \delta^s V \rV = \lV \delta V \rV \llcorner_Z , \quad Z := \Biggl\{ x \in U : \limsup_{ r \to + 0 } \frac{ \lV \delta V \rV ( B_r ( x ) ) }{ \lV V \rV ( B_r ( x ) ) } = \infty \Biggr\}
	\]
	as stated in \cite[Theorem 4.7]{simon1983lectures}. We call the above $ H_V $ the generalized mean curvature of $ V $. If $ V \in \IV_n ( U ) $, by Brakke's perpendicularity theorem \cite[Chapter 5]{brakke1978motion}, $ H_V $ and $ T_x \lV V \rV $ are orthogonal for $ \lV V \rV $-almost everywhere.
	\vskip.3\baselineskip
	Here, we list simple approximation facts to prove the compactness theorem proved in \cite[Proposition 3.4]{Edelen+2020+95+137}.
	
	\begin{lemma}
		\label{density_of_C2t}
		Let $ U \subset \R^{ n + 1 } $ be an open set and let $ M \subset U $ be a domain with $ C^3 $ boundary. We then have the following:
		\begin{enumerate}
			\item the space $ \{ \phi \in C^2_c ( U ; \R ) : \D \phi ( x ) \in T_x ( \partial M ) \text{ for all } x \in \partial U \} $ is dense in $ C^0 ( U ; \R ) $;
			\item if $ \mu $ is finite and $ n $-rectifiable Radon measure on $ U $, and $ 1 \leq p < \infty $, then the $ L^p ( \mu ) $-closure of $ \mathcal{ T }_{ \partial M } C^1_c ( U ; \R^{ n + 1 } ) $ is
			\[
			\bigl\{ X \in L^p ( U , \mu ; \R^{ n + 1 } ) : X ( x ) \in T_x ( \partial M ) \text{ for } \mu \text{-}a.e.\,x \in \partial M \bigr\}.
			\]
		\end{enumerate}
	\end{lemma}
	
	\subsection{Varifold with contact angle}
	Let $ M \subset \R^{ n + 1 } $ be a compact domain with $ C^3 $ boundary, $ \theta \in ( 0 , \pi ) $, $ a = \cos \theta $ and $ \Omega \subset M $ be an open set such that $ \partial \Omega \cap M^{ \circ } $ and $ \partial \Omega \cap \partial M $ are of $ C^3 $ class with $ C^2 $ boundary $ \Gamma ( \Omega ) $. To formulate the contact angle in a variational sense, consider the first variation of the capillary free energy. By calculating the first variation of the capillary free energy
	\[
	F_a ( \Omega ) := \mathcal{ H }^n ( \partial \Omega \cap M^{ \circ } ) + a \mathcal{ H }^n ( \partial \Omega \cap \partial M )
	\]
	with respect to $ X \in \mathcal{ T }_{ \partial M } C^1_c ( \R^{ n + 1 } ; \R^{ n + 1 } ) $, one obtains
	\[
	\delta F_a ( \Omega ) [ X ] = - \int_{ \partial \Omega \cap M^{ \circ } } X \cdot H_{ \partial \Omega }\,d \mathcal{ H }^n + \int_{ \Gamma ( \Omega ) } ( \eta_{ \partial \Omega \cap M^{ \circ } } + a \eta_{ \partial \Omega \cap \partial M } ) \cdot X\,d \mathcal{ H }^{ n - 1 },
	\]
	where $ H_{ \partial \Omega } $ is the mean curvature vector of $ \partial \Omega \cap M^{ \circ } $, $ \eta_{ \partial \Omega \cap M^{ \circ } } $ is the exterior unit co-normal vector of $ \partial \Omega \cap M^{ \circ } $, and $ \eta_{ \partial \Omega \cap \partial M } $ is the exterior unit co-normal vector of $ \partial \Omega \cap \partial M $. If the second term on the right-hand side vanishes, one finds that $ \eta_{ \partial \Omega \cap M^{ \circ } } + a \eta_{ \partial \Omega \cap \partial M } $ is orthogonal to $ \partial M $ on $ \Gamma ( \Omega ) $. For any $ x \in \Gamma ( \Omega ) $, this fact implies that
	\[
	P_{ T_x ( \partial M ) } ( \eta_{ \partial \Omega \cap M^{ \circ } } ( x ) ) = - a \eta_{ \partial \Omega \cap \partial M } ( x ),
	\]
	which further implies that $ \nu_{ \partial M } ( x ) \cdot \eta_{ \partial \Omega \cap M^{ \circ } } ( x ) = \sin \theta $, where $ \nu_{ \partial M } ( x ) $ is the exterior unit normal vector of $ \partial M $ at $ x $. The following Figure \ref{balance} illustrates the above discussion.
	
	\begin{figure}[h]
		\centering
		\includegraphics[width=0.5\linewidth]{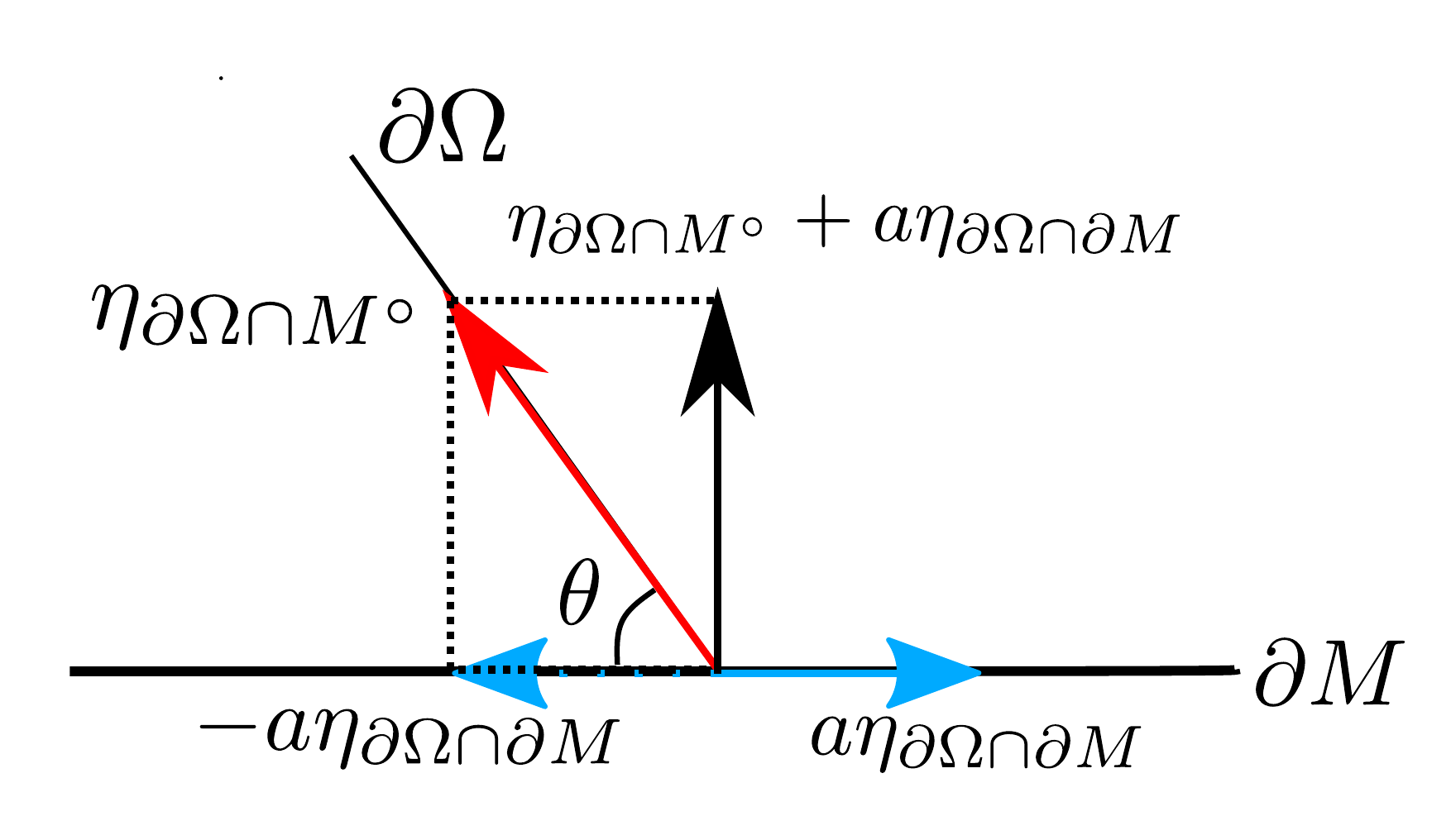}
		\caption{}
		\label{balance}
	\end{figure}
	
	Motivated by this discussion and setting $ a = \cos \theta $, one can define a contact angle condition for varifolds as follows (see, for example, \cite[Section 3.1]{kagaya2017contactangle} and \cite[Section 2.2]{de2022existence} for further details).
	
	\begin{definition}
		Let $ U \subset \R^n $ be an open set, and let $ M \subset U $ be a (relative) closed domain with $ C^3 $ boundary. Let $ ( V , W ) \in \V_n ( M ) \times \V_n ( \partial M ) $ and suppose $ \theta \in ( 0 , \pi ) $. We say that {\itshape $ ( V , W ) $ has the contact angle $ \theta $ in $ M $} if there exists a $ \lV V \rV $-measurable vector field $ H_V \in L^1 ( M , \lV V \rV ; \R^{ n + 1 } ) $ satisfying $ H_V ( x ) \in T_x ( \partial M ) $ for $ \lV V \rV $-almost every $ x \in \partial M $ such that
		\begin{equation}
			\label{contact_angle_condition}
			\int_{ G_n ( M ) } \dv_S X ( x )\,d V ( x , S ) + a \int_{ G_n ( \partial M ) } \dv_S X ( x )\,d W ( x , S ) = - \int_M X \cdot H_V\,d \lV V \rV
		\end{equation}
		for any $ X \in \mathcal{ T }_{ \partial M } C^1_c ( U ; \R^{ n + 1 } ) $, where $ a = \cos \theta $. In particular, we say that {\itshape $ V $ has free-boundary} if \eqref{contact_angle_condition} holds in the case $ a = 0 $. We call $ H_V $ the generalized mean curvature vector of $ V $ with the constant angle $ \theta $.
	\end{definition}
	
	We quote the following estimate for contact angle varifolds in \cite[Corollary 3.18]{de2022existence}.
	
	\begin{prop}
		\label{boundedness_first_variation}
		Let $ ( V , W ) \in \V_n ( M ) \times \V_n ( \partial M ) $ the contact angle $ \theta \in ( 0 , \pi / 2 ) $ in $ M $ and let $ a = \cos \theta $. Then $ \lV \delta V + a \delta W \rV $ is locally finite in $ U $, and for every open set $ K \subset \subset K' \subset \subset U $, we have
		\begin{equation}
			\label{boundedness_first_ineq}
			\lV \delta V + a \delta W \rV ( K ) \leq C ( n ) \int_{ K' } \lv H_V \rv\,d \lV V \rV + C' ( n , K , K' ) \lV V \rV ( K' )
		\end{equation}
		for some $ C ( n ) > 0 $ and $ C' ( n , K , K' ) > 0 $.
	\end{prop}
	
	\begin{remark}
		\label{oscillation}
		In general, even if $ ( V , W ) \in \V_n ( M ) \times \V_n ( \partial M ) $ has the contact angle $ \theta $, $ \delta V $ and $ \delta W $ may not be locally finite. For example, let $ U \subset \{ 0 \} \times \R^n $ be a relative open set with infinite perimeter in $ \{ 0 \} \times \R^n $. Since any closed set is the $ 0 $-level set of a smooth function, there exists a smooth $ n $-dimensional surface $ \Sigma \subset ( 0 , \infty ) \times U $ such that $ \Sigma $ is locally a graph on $ \{ 0 \} \times \R^n $ near the boundary $ \partial U $ and meets $ \{ 0 \} \times \R^n $ tangentially on $ \partial U $. Let us consider the varifold $ V $ induced by $ \Sigma $ and $ W $ by $ \{ 0 \} \times \R^n \setminus U $, then $ ( V , 2 W ) $ has the contact angle $ \pi / 3 $, that is, $ a = 1 / 2 $. However, $ \delta V $ and $ \delta W $ are not locally finite. Therefore the handling of the first variations of $ V $ and $ W $ requires caution.
	\end{remark}
	
	Here, we introduce the results in \cite[Lemma 3.6, Corollary 3.9]{de2022existence}, for the purposes of this paper. By contrast with the above remark, one can understand the structure as varifolds on the boundary if it has a bounded variation with respect to $ \mathcal{ T }_{ \partial M } C^1_c $.
	
	\begin{prop}
		\label{constancy}
		Let $ V \in \V_n ( U ) $ have the bounded first variation $ \delta V $ with respect to $ \mathcal{ T }_{ \partial M } C^1_c ( U ; \R^{ n + 1 } ) $. Then we have
		\begin{equation}
			V ( \{ ( x , S ) \in G_n ( U ) : x \in \partial M , S \neq T_x ( \partial M ) \} ) = 0.
		\end{equation}
		Moreover, $ V \llcorner_{ G_n ( \partial M ) } $ is an $ n $-rectifiable varifold and $ \lV V \rV = \Theta^n ( \lV V \rV , \cdot )\,\mathcal{ H }^n \llcorner_{ \partial M } $.
	\end{prop}
	
	\subsection{Definition of weak MCF}
	We begin with the introduction to the definition of a (modified) weak MCF with contact angle in the sense of Brakke \cite{brakke1978motion}.
	\begin{definition}
		\label{def_of_Brakke}
		Let $ U \subset \R^n $ be an open set, let $ M \subset U $ be a (relative) closed domain with $ C^3 $ boundary $ \partial M $, $ \theta \in ( 0 , \pi / 2 ] $ and $ a = \cos \theta $. A family of varifolds $ \{ ( V_t , W_t ) \}_{ t \geq 0 } $ is {\itshape an $ n $-dimensional Brakke flow with a contact angle $ \theta $ in $ M \subset U $} if all of the following hold:
		\begin{enumerate}
			\item for almost every time $ t \geq 0 $, $ ( V_t , W_t ) $ is of $ \IV_n ( M ) \times \IV_n ( \partial M ) $ and it has the contact angle $ \theta $ with having the $ L^2 $ mean curvature $ H \in L_{ loc }^1 ( [ 0 , \infty ) ; L^2_{ loc } ( M , \lV V_t \rV ; \R^{ n + 1 } ) ) $ with $ H ( x , t ) \in T_x ( \partial M ) $ for almost every $ t \geq 0 $ and $ \mathcal{ H }^n $-almost every $ x \in \partial M $;
			\item for all $ T > 0 $ and all compact set $ K \subset M $, $ \sup_{ t \in [ 0 , T ] } ( \lV V_t \rV + \lV W_t \rV ) ( K ) < \infty $;
			\item for all $ 0 \leq t_1 < t_2 < \infty $ and all test function $ \phi \in C^1_c ( U \times [ 0 , \infty ) ; [ 0 , \infty ) ) $ with $ \D \phi ( \cdot , t ) \in \mathcal{ T }_{ \partial M } C^0_c ( U ; \R^{ n + 1 } ) $ all $ t > 0 $,
			\begin{equation}
				\label{Brakkeineq}
				\begin{split}
					&( \lV V_{ t_2 } \rV + a \lV W_{ t_2 } \rV ) ( \phi ( \cdot , t_2 ) ) - ( \lV V_{ t_1 } \rV + a \lV W_{ t_1 } \rV ) ( \phi ( \cdot , t_1 ) )\\
					&\leq \int^{ t_2 }_{ t_1 } \int_M ( \D \phi ( x , t ) - \phi ( x , t ) H ( x , t ) ) \cdot H ( x , t )\,d \lV V_t \rV d t + \pt \phi ( x , t )\,d ( \lV V_t \rV + a \lV W_t \rV ) d t.
				\end{split}
			\end{equation}
			We call inequality \eqref{Brakkeineq} Brakke's inequality.
		\end{enumerate}
	\end{definition}
	
	Including the boundary measure $ W_t $ in inequality \eqref{Brakkeineq} may seem unnatural at first. However, there are several reasons for incorporating the boundary measure $ W_t $ into the motion law \eqref{Brakkeineq} of the interface. The first reason is that the motion law of the internal interface $ V_t $ should describe the evolution of the boundary measure $ W_t $, which is driven by physical insights. The second is to prove a compactness property via the definition. In the proof, we argue that the limit measure of varifolds is uniquely determined, regardless of the choice of subsequences independent of time, by utilizing the fact that a Brakke flow has a semi-decreasing property with respect to time. Therefore, to establish this monotonicity including the boundary measure, $ W_t $ must be incorporated into inequality \eqref{Brakkeineq}. Another reason is to avoid the redundant situation where $ W_t $ suddenly becomes the entire boundary measure at some point in time. Specifically, we need to avoid the situation that $ W_t = \mathcal{ H }^n \llcorner_{ \partial M } $ at some time. If this were to occur, \eqref{contact_angle_condition} would coincide with the $ \ang{ 90 } $ angle condition, and it would no longer be appropriate to refer to it as the contact angle condition.
	
	\subsection{Main results}
	\label{sec_mainresult}
	We now present an existence theorem for a weak notion of MCF with a prescribed contact angle when given an initial datum $ E_0 $. Since the $ \ang{ 90 } $ angle condition has already been studied by Edelen \cite{Edelen+2020+95+137}, we only consider the case where $ \theta \neq \pi / 2 $ in this paper.
	
	\begin{theorem}
		\label{main_results}
		Let $ M \subset \R^{ n + 1 } $ be a compact domain with $ C^3 $ boundary, $ E_0 \subset M $ be a set of finite perimeter, $ \theta \in ( 0 , \pi / 2 ) $, and $ a = \cos \theta $. Then there exists a Brakke flow $ \{ ( V_t , W_t ) \}_{ t \geq 0 } $ with the contact angle $ \theta $ starting with the reduced boundary $ \partial^* E_0 $. Moreover, there exists a set of locally finite perimeter $ E \subset M \times [ 0 , \infty ) $ such that the following hold;
		\begin{enumerate}
			\item $ ( V_t , W_t ) $ satisfies
			\[
			\lV V_0 \rV + a \lV W_0 \rV = \lim_{ t \to 0^+ } ( \lV V_t \rV + a \lV W_t \rV ) = \mathcal{ H }^n \llcorner_{ M^{ \circ } \cap \partial^* E_0 } + a \mathcal{ H }^n \llcorner_{ \partial M \cap \partial^* E_0 },
			\]
			and $ W_t $ is a unit density $ n $-varifold for almost every $ t > 0 $;
			\item the characteristic function $ \chi_E $ satisfies the following properties:
			\begin{enumerate}
				\item $ \chi_E $ is of $ C^{ \frac{ 1 }{ 2 } }_{ loc } ( [ 0 , \infty ) ; L^1 ( M ) ) $ and $ \chi_E ( \cdot , 0 ) = \chi_{ E_0 } ( \cdot ) $ for almost everywhere;
				\item $ \lV \D \chi_{ E_t } \rV \llcorner_{ M^{ \circ } } + a \lV \D \chi_{ E_t } \rV \llcorner_{ \partial M } \leq \lV V_t \rV + a \lV W_t \rV $ for all $ t \geq 0 $;
				\item for all Borel set $ I \subset [ 0 , \infty ) $,
				\[
				\lV \D' \chi_E \rV ( M \times I ) \leq \frac{ 1 }{ a } \bigl( \mathcal{ L }^1 ( I ) + \mathcal{ L }^1 ( I )^{ \frac{ 1 }{ 2 } } \bigr) \mathcal{ H }^n ( \partial^* E_0 ).
				\]
			\end{enumerate}
		\end{enumerate}
	\end{theorem}
	
	\begin{definition}
		Let $ k $ be an integer. For a set $ A \subset \R^k \times \R $, we define the slice of $ A $ at $ z \in \R $ by $ A_z := \{ x \in \R^k : ( x , z ) \in A \} $.
	\end{definition}
	
	In the main theorem, it may appear that $ \chi_E $ is redundant. However, it has a few important roles besides $ \chi_E $ emerges naturally from the approach of this paper. First, $ \chi_E $ guarantees that there will be no superfluous non-uniqueness issue of Brakke-type MCFs such as sudden disappearance. Because, since the time-slice volume of $ \chi_E $ is continuous in $ L^1 $ with respect to time, the surface $ \partial^* E_t $ cannot vanish instantaneously for any time even under the weak solution setting. Second, the existence of $ \chi_E $ restricts the possible singularities of the interface $ \lV V_t \rV $. For example, in the $ n = 2 $ case, one can see that a unit density $ \lV V_t \rV $ cannot form a triple junction since $ \partial^* E_t $ cannot have a triple junction. The third role of $ \chi_E $ is that it ensures that vanishing the boundary measure $ W_t $ does not occur. Even if $ W_t $ suddenly disappears at some time, $ V_t $ may continue to evolve in time as if it satisfies the Brakke inequality \eqref{Brakkeineq} and $ ( V_t , W_t ) $ satisfies \eqref{contact_angle_condition} corresponding to the $ \ang{ 90 } $ angle condition. Such a set $ E $ appears in the almost same roles in \cite{takasao2016existence} as well. We note that the boundary measure of $ E $ also plays a role in the concept of enhanced motion proposed by Ilmanen \cite[Section 8]{ilmanen1994elliptic}.
	
	\section{Compactness for varifolds with contact angle}
	\label{cpt_contactangle_varifolds}
	In this section, we prove a compactness theorem for a sequence of contact angle varifolds $ V^i + a W^i $ under the assumption that $ \delta V^i $ and $ \delta W^i $ are uniformly locally bounded. Examples mentioned in Remark \ref{oscillation} necessitate this condition. Given this strong assumption, the proof of compactness proceeds by the standard compactness theorem for integral varifolds and measure-function pairs by Hutchinson's study \cite[Theorem 4.4.2]{hutchinson1986second}.
	
	\begin{theorem}
		\label{compactness_for_contact_angle}
		Let $ \theta \in ( 0 , \pi / 2 ) $ and $ a = \cos \theta $. Let $ U^i \subset \R^{ n + 1 } $ be open sets and $ M^i \subset U^i $ be (relative) closed sets with $ C^3 $ boundary, and let $ \{ ( V^i , W^i ) \}_{ i \in \N } \subset \IV_n ( M^i ) \times \IV_n ( \partial M^i ) $ have the contact angle $ \theta $ with the generalized mean curvature $ H_{ V^i } $ in $ M^i \subset U^i $. Suppose $ U^i \to U $, $ \partial M^i \to \partial M $ in $ C^3_{ loc } $, and for all compact set $ K \subset U $,
		\begin{equation}
			\label{contact_cpt_assumption}
			\sup_i \Biggl( \int_K | H_{ V^i } |^2\,d \lV V^i \rV + ( \lV V^i \rV + a \lV W^i \rV + \lV \delta V^i \rV + a \lV \delta W^i \rV ) ( K ) \Biggr) \leq C ( K ),
		\end{equation}
		for some $ C ( K ) > 0 $. Then there exists a pair of varifolds $ ( V , W ) \in \IV_n ( M ) \times \IV_n ( \partial M ) $ with the contact angle $ \theta $ such that, taking a subsequence if necessary, $ ( V^i , W^i ) \rightharpoonup ( V , W ) $ as varifolds and $ \delta V $ and $ \delta W $ are locally bounded. Moreover, if $ X^i \in \mathcal{ T }_{ \partial M^i } C^1_c ( U^i ; \R^{ n + 1  } ) $ converge to $ X \in \mathcal{ T }_{ \partial M } C^1_c ( U ; \R^{ n + 1 } ) $ in $ C^1 $ with uniformly bounded supports, then
		\[
		\lim_{ i \to \infty } \int_{ U^i } H_{ V^i } \cdot X^i\,d \lV V^i \rV = \int_U H_V \cdot X\,d \lV V \rV
		\]
		and hence, for all $ \phi \in C^0_c ( U ; [ 0 , \infty ) ) $,
		\[
		\int_U \phi \lv H_V \rv^2\,d \lV V \rV \leq \liminf_{ i \to \infty } \int_{ U^i } \phi \lv H_{ V^i } \rv^2\,d \lV V^i \rV.
		\]
	\end{theorem}
	
	\begin{remark}
		The inclusion of the $ L^2 $ norm of $ H_{ V^i } $ in \eqref{contact_cpt_assumption} may seem excessive. However, even if the first variation is uniformly bounded, it remains unclear whether the weak convergence of $ H_{ V^i }\,d \lV V^i \rV $ can be interpreted solely as an integral of $ V $, as in the free boundary case. This is because we impose a strong assumption to apply \cite[Theorem 4.4.2]{hutchinson1986second} in this context.
	\end{remark}
	
	\section{Compactness for Brakke flows with contact angle}
	\label{cpt_contactangle_Brakke}
	In this section, we prove the compactness theorem for contact angle Brakke flows. To this end, we show the basic properties of Brakke flow and a technical lemma to deal with the measure $ W_t $ on the boundary $ \partial M $. First, we prove the following semi-decreasing property, which is crucial for the Brakke flow.
	
	\begin{lemma}
		\label{Basic_props}
		Let $ \{ ( V_t , W_t ) \}_{ t \geq 0 } $ be a Brakke flow with the contact angle $ \theta $ in $ M \subset U $. Suppose, for all compact set $ K \subset M $, 
		\[
		\sup_{ t \geq 0 } ( \lV V_t \rV + a \lV W_t \rV ) ( K ) \leq C_0 ( K )
		\]
		for some $ C_0 ( K ) > 0 $. Then all of the following hold:
		\begin{enumerate}
			\item for each $ \phi \in C^2_c ( U ; [ 0 , \infty ) ) $ with $ \D \phi \in \mathcal{ T }_{ \partial M } C^1_c ( U ; \R^{ n + 1 } ) $, there exists a constant $ C_1 = C_1 ( C_0 , \phi ) > 0 $ such that
			\begin{equation*}
				\label{monotone_Brakke}
				t \mapsto ( \lV V_t \rV + a \lV W_t \rV ) ( \phi ) - C_1 t
			\end{equation*}
			is decreasing in $ t \geq 0 $;
			\item the left-limit and the right-limit of $ \lV V_t \rV + a \lV W_t \rV $ exist at every $ t \geq 0 $ and satisfy
			\begin{equation*}
				\label{conti_Brakke}
				\lim_{ s \to t^+ } ( \lV V_s \rV + a \lV W_s \rV ) \leq \lV V_t \rV + a \lV W_t \rV \leq \lim_{ s \to t^- } ( \lV V_s \rV + a \lV W_s \rV );
			\end{equation*}
			\item for each $ 0 \leq t_1 < t_2 < \infty $ and $ \phi \in C^0_c ( U ; [ 0 , \infty ) ) $, take a compact set $ K \subset U $ satisfying $ \supp \phi \subset K $, then there exists a constant $ C_2 = C_2 ( C_0 ( K ), \phi , t_1 , t_2 ) > 0 $ such that
			\begin{equation}
				\label{L2MC_bound_Brakke}
				\int^{ t_2 }_{ t_1 } \int_M \phi \lv H \rv^2\,d \lV V_t \rV d t \leq C_2.
			\end{equation}
		\end{enumerate}
	\end{lemma}
	
	\begin{proof}
		For all $ 0 \leq t_1 < t_2 < \infty $ and all $ \phi \in C^2_c ( U ; [ 0 , \infty ) ) $ with $ \D \phi \in\mathcal{ T }_{ \partial M } C^1_c ( U ; \R^{ n + 1 } ) $, by \eqref{Brakkeineq}, we have
		\begin{equation}
			\label{monotone_Brakke_flows}
			\begin{split}
				&\int_M \phi\,d ( \lV V_{ t_2 } \rV + a \lV W_{ t_2 } \rV ) - \int_M \phi\,d ( \lV V_{ t_1 } \rV + a \lV W_{ t_1 } \rV )
				\leq \int^{ t_2 }_{ t_1 } \int_M - \phi \lv H \rv^2 + H \cdot \D \phi\,d \lV V_t \rV d t\\
				&\leq \int^{ t_2 }_{ t_1 } \int_M - \frac{ 1 }{ 2 } \phi \lv H \rv^2 + \frac{ \lv \D \phi \rv^2 }{ 2 \phi }\,d \lV V_t \rV d t
				\leq \int^{ t_2 }_{ t_1 } \int_M - \frac{ 1 }{ 2 } \phi \lv H \rv^2\,d \lV V_t \rV d t + C_0 ( \supp \phi ) C ( \phi ) ( t_2 - t_1 ),
			\end{split}
		\end{equation}
		where we used the Cauchy--Schwarz inequality, the fact that
		\[
		\sup_{ \{ \phi > 0 \} } \frac{ \lv \D \phi \rv^2 }{ \phi } \leq 2 \sup_U \lV \D^2 \phi \rV ( =: C ( \phi ) )
		\]
		for every $ \phi \in C^2_c ( U ; [ 0 , \infty ) ) $, and the assumption of this lemma. Thus we have
		\[
		( \lV V_{ t_2 } \rV + a \lV W_{ t_2 } \rV ) ( \phi ) - C_1 t_2 \leq ( \lV V_{ t_1 } \rV + a \lV W_{ t_1 } \rV ) ( \phi ) - C_1 t_1,
		\]
		where $ C_1 = C_0 ( \supp \phi ) C ( \phi ) $. This completes the proof of (1), and part (3) follows from \eqref{monotone_Brakke_flows} and taking a compact set $ K \supset \supp \phi $. By part (1), the left-/right-limits exist for each $ t > 0 $ and we have
		\begin{equation*}
			\begin{split}
				( \lV V_{ t + \varepsilon } \rV + a \lV W_{ t + \varepsilon } \rV ) ( \phi ) - C_1 \varepsilon
				&\leq ( \lV V_t \rV + a \lV W_t \rV ) ( \phi )\\
				&\leq ( \lV V_{ t - \varepsilon } \rV + a \lV W_{ t - \varepsilon } \rV ) ( \phi ) + C_1 \varepsilon,
			\end{split}
		\end{equation*}
		for each $ \phi \in C^2_c ( U ; [ 0 , \infty ) ) $ with $ \D \phi \in \mathcal{ T }_{ \partial M } C^1_c ( U ; \R^{ n + 1 } ) $ and sufficiently small $ \varepsilon > 0 $. Therefore, by limiting $ \varepsilon \to 0 $ and the approximation from Lemma \ref{density_of_C2t}, we obtain
		\[
			\lim_{ s \to t^+ } ( \lV V_s \rV + a \lV W_s \rV ) ( \phi ) \leq ( \lV V_t \rV + a \lV W_t \rV ) ( \phi ) \leq \lim_{ s \to t^- } ( \lV V_s \rV + a \lV W_s \rV ) ( \phi )
		\]
		for any $ \phi \in C^0_c ( U ; [ 0 , \infty ) ) $ and $ t \geq 0 $.
	\end{proof}
	
	Next, we prove the following technical lemma. It is necessary to determine the limit measure $ W_t $ independent of the choice of subsequence for time.
	
	\begin{lemma}
		\label{W_dens_lemma}
		Let $ U $ and $ U^i $ be open sets so that $ U^i \to U $ in $ C^3_{ loc } $. Let $ M \subset U $ and $ M^i \subset U^i $ have $ C^3 $ boundary such that $ \partial M^i \to \partial M $ in $ C^3_{ loc } $. Let $ W^i \in \IV_n ( \partial M^i ) $ and $ W \in \IV_n ( \partial M ) $ be a family of varifolds with the bounded first variation with respect to $ \mathcal{ T }_{ \partial M^i } C^1_c ( U^i ; \R^{ n + 1 } ) $ and $ \mathcal{ T }_{ \partial M } C^1_c ( U ; \R^{ n + 1 } ) $, respectively. Suppose that, for some positive integer $ q $,
		\[
		\lV W^i \rV ( \{ \Theta^n ( \lV W^i \rV , x ) \geq q + 1 \} ) = 0 \text{ for each } i, \quad \lim_{ i \to \infty } W^i = W \text{ as varifolds},
		\]
		where $ \Theta^n ( \lV W \rV , x ) = \lim_{ r \to + 0 } \lV W \rV ( B_r ( x ) ) / \omega_n r^n $. Then $ \lV W \rV ( \{ \Theta^n ( \lV W \rV , x ) \geq q + 1 \} ) = 0 $.
	\end{lemma}
	
	\begin{proof}
		By Lemma \ref{constancy}, we can write down $ W^i $ and $ W $ as
		\[
		\lV W^i \rV = \Theta^n ( \lV W^i \rV , \cdot )\,\mathcal{ H }^n \llcorner_{ \partial M^i } \text{ and } \lV W \rV = \Theta^n ( \lV W \rV , \cdot )\,\mathcal{ H }^n \llcorner_{ \partial M }.
		\]
		Assume that $ \lV W \rV ( \{ \Theta^n ( \lV W \rV , x ) \geq q + 1 \} ) > 0 $ by a contradiction. Let $ K \subset \{ \Theta^n ( \lV W \rV , x ) \geq q + 1 \} $ be a compact set satisfying $ \lV W \rV ( K ) > 0 $, that is, $ \mathcal{ H }^n ( K ) > 0 $. Then, by the upper semi-continuity of Radon measures, we have
		\begin{equation*}
			\begin{split}
				( q + 1 ) \mathcal{ H }^n ( K ) &\leq \int_K \Theta^n ( \lV W \rV , x )\,d \mathcal{ H }^n \llcorner_{ \partial M } = \lV W \rV ( K ) \leq \limsup_{ i \to \infty } \lV W^i \rV ( K )\\
				&= \limsup_{ i \to \infty } \int_K \Theta^n ( \lV W^i \rV , x )\,d \mathcal{ H }^n \llcorner_{ \partial M^i } \leq q \limsup_{ i \to \infty } \int_K d \mathcal{ H }^n \llcorner_{ \partial M^i } = q \mathcal{ H }^n ( K ).
			\end{split}
		\end{equation*}
		This is a contradiction. Thus the claim is proven.
	\end{proof}
	
	The above two results allow us to prove the following compactness property for Brakke flows with contact angle under the almost same assumption \eqref{contact_cpt_assumption}.
	
	\begin{theorem}
		\label{compactness_Brakke}
		Let $ \theta \in ( 0 , \pi / 2 ) $, $ a = \cos \theta $, $ U^i \subset \R^{ n + 1 } $ be open sets, $ M^i \subset U^i $ be (relative) closed sets with $ C^3 $ boundary, and let $ \{ ( V_t^i , W_t^i ) \}_{ t \geq 0 } $ be a sequence of Brakke flows with the contact angle $ \theta $ in $ M^i \subset U^i $. Suppose that $ U^i \to U $, $ \partial M^i \to \partial M $ in $ C^3_{ loc } $, for any compact set $ K \subset M $ $ ( V_t^i , W_t^i ) $ satisfy
		\begin{equation}
			\label{mass_bound_assumption_for_Brakke_flow}
			\sup_i \sup_{ t \geq 0 } ( \lV V_t^i \rV + a \lV W_t^i \rV ) ( K ) \leq D_1 ( K ) < \infty
		\end{equation}
		for some $ D_1 ( K ) > 0 $, and $ W^i_t $ is a unit density varifold for each $ i \in \N $ and for almost every $ t \geq 0 $. Furthermore, for every open set $ K \subset \subset K^{ \prime } \subset \subset U $, there exists a constant $ D_2 ( K , K^{ \prime } ) > 0 $ such that, for almost every $ t > 0 $ and all $ i $,
		\begin{equation}
			\label{mass_bound_assumption_for_Brakke_flow2}
			( \lV \delta V_t^i \rV + a \lV \delta W_t^i \rV ) ( K ) \leq D_2 ( K , K^{ \prime } ) \int_{ K^{ \prime } } ( 1 + \lv H_{ V_t^i } \rv )\,d ( \lV V_t^i \rV + a \lV W_t^i \rV ).
		\end{equation}
		Then there exist a subsequence $ \{ i_j \}_{ j \in \N } $ and a Brakke flow $ \{ ( V_t , W_t ) \}_{ t \geq 0 } $ with the contact angle $ \theta $ such that for all $ t \geq 0 $ we have
		\[
		\lim_{ j \to \infty } \lV V_t^{ i_j } \rV = \lV V_t \rV , \quad \lim_{ j \to \infty } ( \lV V_t^{ i_j } \rV + a \lV W_t^{ i_j } \rV ) = \lV V_t \rV + a \lV W_t \rV
		\]
		as Radon measures on $ M^{ \circ } $ and $ M $, respectively. Moreover, for almost every time $ t \geq 0 $ there exists a subsequence $ \{ i'_j \}_{ j \in \N } \subset \{ i_j \}_{ j \in \N } $ such that
		\[
		\lim_{ j \to \infty } V_t^{ i'_j } = V_t , \quad \lim_{ j \to \infty } W_t^{ i'_j } = W_t
		\]
		as varifolds on $ M $.
	\end{theorem}
	
	\begin{proof}
		Let $ \{ \phi_k \}_{ k \in \N } \subset C^2_c ( U ; [ 0 , \infty ) ) $ with $ \D \phi_k \in \mathcal{ T }_{ \partial M } C^1_c ( U ; \R^{ n + 1 } ) $ be a countable set of test functions such that it is dense in $ C^0_c ( U ; [ 0 , \infty ) ) $ with respect to the sup norm $ \sup \lv \cdot \rv $. We can take such a set by Lemma \ref{density_of_C2t}. As in Lemma \ref{monotone_Brakke_flows} (1), for each $ i \in \N $ and $ \phi_k $, $ f_{ i , k } ( t ) := ( \lV V^i_t \rV + a \lV W^i_t \rV ) ( \phi ) - C_1 t $ is a monotone decreasing function in $ t \geq 0 $. By \eqref{mass_bound_assumption_for_Brakke_flow}, $ \{ f_{ i , k } ( t ) \}_{ i \in \N } $ is locally bounded for each $ i , k \in \N $ and $ t \geq 0 $. Thus, by Helly's selection principle and a diagonal argument, there exist a subsequence $ \{ i_j \}_{ j \in \N } \subset \N $ and a decreasing function $ g_k ( t ) $ such that, for all $ t \geq 0 $, 
		\[
		\lim_{ j \to \infty } f_{ i_j , k } ( t ) = g_k ( t ).
		\]
		Therefore $ \lim_j ( \lV V^{ i_j }_t \rV + a \lV W^{ i_j }_t \rV ) ( \phi_k ) $ exists for every $ t \geq 0 $ and $ \phi_k $. Hence, by the density argument, $ ( \lV V^{ i_j }_t \rV + a \lV W^{ i_j }_t \rV ) ( \phi ) $ is also convergent for any fixed $ \phi \in C^0_c ( U ; [ 0 , \infty ) ) $, and the limit defines a locally bounded linear functional on $ C^0_c ( U ; [ 0 , \infty ) ) $ by \eqref{mass_bound_assumption_for_Brakke_flow}. By the Riesz representation theorem, there exists a Radon measure $ \mu_t $ on $ U $ such that
		\begin{equation}
			\label{conv_mu}
			\lim_{ j \to \infty } ( \lV V^{ i_j }_t \rV + a \lV W^{ i_j }_t \rV ) ( \phi ) = \mu_t ( \phi )
		\end{equation}
		for all $ t \geq 0 $ and $ \phi \in C^0_c ( U ; [ 0 , \infty ) ) $.
		\vskip.3\baselineskip
		We next prove that $ \mu_t $ determines a pair of varifolds $ ( V_t , W_t ) \in \IV_n ( M ) \times \IV_n ( \partial M ) $ with the contact angle $ \theta $ as the convergence. Thanks to \eqref{L2MC_bound_Brakke}, \eqref{mass_bound_assumption_for_Brakke_flow} and Fatou's Lemma, for almost every time $ t \geq 0 $ and all $ \phi \in C^0_c ( U ; [ 0 , \infty ) ) $, we have
		\begin{equation}
			\label{L2MC_bound_a.e.}
			\liminf_{ j \to \infty } \int_U \phi \lv H_{ V^{ i_j }_t } ( x ) \rv^2\,d \lV V^{ i_j }_t \rV < \infty.
		\end{equation}
		From \eqref{mass_bound_assumption_for_Brakke_flow}, \eqref{mass_bound_assumption_for_Brakke_flow2}, \eqref{L2MC_bound_a.e.} and the H\"{o}lder inequality, we can apply Theorem \ref{compactness_for_contact_angle} to the family $ \{ ( V^{ i_j }_t , W^{ i_j }_t ) \}_{ j \in \N } $ at almost every time $ t \geq 0 $ so that there exist a subsequence $ \{ i'_j \}_{ j \in \N } \subset \{ i_j \}_{ j \in \N } $ and $ ( V_t , W_t ) \in \IV_n ( M ) \times \IV_n ( \partial M ) $ such that $ \lim_j V^{ i'_j }_t = V_t $, $ \lim_j W^{ i'_j }_t = W_t $ as varifolds and $ ( V_t , W_t ) $ has the contact angle $ \theta $. Note that the choice of the subsequence may depend on $ t $. On the other hand, by \eqref{conv_mu}, we have
		\begin{equation*}
			\begin{split}
				\mu_t \llcorner_{ \partial M } &= ( \lV V_t \rV + a \lV W_t \rV ) \llcorner_{ \partial M },\\
				\mu_t \llcorner_{ M^{ \circ } } &= \lV V_t \rV \llcorner_{ M^{ \circ } },
			\end{split}
		\end{equation*}
		for almost every time $ t \geq 0 $. By Theorem \ref{compactness_for_contact_angle}, we see the boundedness of the first variations of $ V_t $ and $ W_t $. Applying Lemma \ref{constancy} to $ ( V_t , W_t ) $ and using the fact for integral varifolds, there exist a countably $ n $-rectifiable set $ \Gamma_t \subset M^{ \circ } $ and a function $ \Theta_t : \Gamma_t \to \N $ such that
		\begin{equation}
			\label{mut_explicit}
			\begin{split}
				\mu_t \llcorner_{ \partial M } &= ( \lV V_t \rV + a \lV W_t \rV ) \llcorner_{ \partial M } = \Theta^n ( \cdot , t )\,\mathcal{ H }^n \llcorner_{ \partial M },\\
				\mu_t \llcorner_{ M^{ \circ } } &= \lV V_t \rV \llcorner_{ M^{ \circ } } = \Theta_t\,\mathcal{ H }^n \llcorner_{ \Gamma_t },
			\end{split}
		\end{equation}
		where $ \Theta^n ( x , t ) = \Theta^n ( \lV V_t \rV + a \lV W_t \rV , x ) $. In particular, since Lemma \ref{W_dens_lemma} and $ W^i_t $ is a unit density varifold for each $ i $ and for almost every $ t \geq 0 $ due to the assumption, we may assume that $ W_t $ is a unit density $ n $-varifold. In $ M^{ \circ } $, $ V_t $ is uniquely determined by $ \mu_t $ from \eqref{mut_explicit}. On $ \partial M $, since $ a \in ( 0 , 1 ) $, $ W_t $ is a unit density $ n $-varifold, and $ V_t \in \IV_n ( M ) $, $ a \lV W^{ i^{ \prime }_j } \rV $ must converge to $ ( \Theta^n ( \cdot , t ) - \lfloor \Theta^n ( \cdot , t ) \rfloor )\,\mathcal{ H }^n \llcorner_{ \partial M } $, where $ \lfloor \cdot \rfloor $ is the nearest integer that is less than or equal to the given number. Hence we see that $ \lfloor \Theta^n ( \cdot , t ) \rfloor \,\mathcal{ H }^n \llcorner_{ \partial M } $ and $ ( \Theta^n ( \cdot , t ) - \lfloor \Theta^n ( \cdot , t ) \rfloor ) / a\,\mathcal{ H }^n \llcorner_{ \partial M } $ are integral $ n $-rectifiable. Thus, for almost every time $ t \geq 0 $, there exists the pair $ ( V_t , W_t ) \in \IV_n ( M ) \times \IV_n ( \partial M ) $ such that
		\begin{equation*}
			\begin{split}
				\lV V_t \rV &= \Theta_t\,\mathcal{ H }^n \llcorner_{ \Gamma_t } + \lfloor \Theta^n ( \cdot , t ) \rfloor\,\mathcal{ H }^n \llcorner_{ \partial M },\\
				a \lV W_t \rV &= ( \Theta^n ( \cdot , t ) - \lfloor \Theta^n ( \cdot , t ) \rfloor )\,\mathcal{ H }^n \llcorner_{ \partial M },
			\end{split}
		\end{equation*}
		and $ ( V_t , W_t ) $ has the constant angle $ \theta $. This implies that $ \mu_t $ determines $ V_t $ and $ W_t $ so that $ \mu_t = \lV V_t \rV + a \lV W_t \rV $. Note that, if we take a subsequence $ \{ i^{ \prime }_j \}_{ j \in \N } \subset \{ i_j \}_{ j \in \N } $ such that $ V^{ i^{ \prime }_j } $ and $ W^{ i^{ \prime }_j } $ converge as integral varifolds, we see that they must converge to the above $ V_t $ and $ W_t $ independently of the choice of such subsequences, respectively (Although it is not guaranteed that $ \lV V_t^{ i_j } \rV $ and $ \lV W_t^{ i_j } \rV $ converge to $ \lV V_t \rV $ and $ \lV W_t \rV $ as Radon measures, respectively, we note that it suffices for the purpose of establishing the existence theorem in this paper). For $ t \geq 0 $ where \eqref{L2MC_bound_a.e.} does not hold, define $ ( V_t , W_t ) \in \V_n ( M ) \times \V_n ( \partial M ) $ so that $ \mu_t = \lV V_t \rV + a \lV W_t \rV $. Note that they may not be rectifiable and may not have the contact angle $ \theta $, but we do not need to care since the measure of such a set of times is zero. Thus we obtain the convergent varifolds $ ( V_t , W_t ) $ for all $ t \geq 0 $ and it has the contact angle $ \theta $ for almost every $ t \geq 0 $.
		\vskip.3\baselineskip
		Next, we prove that the above varifolds $ ( V_t , W_t ) $ satisfy Brakke's inequality \eqref{Brakkeineq}. The following proof follows along \cite[Theorem 3.7]{tonegawa2019brakke} and \cite[Theorem 4.14]{Edelen+2020+95+137} in much the same way, but we will explain for the convenience of the reader. Let $ \phi \in C^2_c ( U \times [ 0 , \infty ) ; [ 0 , \infty ) ) $ with $ \D \phi ( \cdot , t ) \in \mathcal{ T }_{ \partial M } C^1_c ( U ; \R^{ n + 1 } ) $ for all $ t \geq 0 $, and let $ 0 \leq t_1 < t_2 < \infty $ fix arbitrarily. For $ \phi $, we take the approximate sequence $ \phi^{ i_j } \in C^2_c ( U^{ i_j } \times [ 0 , \infty ) ; [ 0 , \infty ) ) $ with $ \D \phi^{ i_j } ( \cdot , t ) \in \mathcal{ T }_{ \partial M^{ i_j } } C^1_c ( U^{ i_j } ; \R^{ n + 1 } ) $ so that $ \phi^{ i_j } \to \phi $ in $ C^2_{ loc } $ with uniformly bounded supports $ K' \subset \R^{ n + 1 } $ and $ \sup_j \sup_{ U^{ i_j } } \lV \D^2 \phi^{ i_j } \rV < \infty $. By \eqref{monotone_Brakke_flows} and \eqref{mass_bound_assumption_for_Brakke_flow}, we have
		\begin{equation*}
			\int_{ M^{ i_j } } - \phi^{ i_j } \lv H_{ V^{ i_j }_t } \rv^2 + H_{ V^{ i_j }_t } \cdot \D \phi^{ i_j }\,d \lV V^{ i_j }_t \rV \leq \int_{ M^{ i_j } } \frac{ \lv \D \phi^{ i_j } \rv^2 }{ 2 \phi^{ i_j } }\,d \lV V^{ i_j }_t \rV \leq C ( \phi ) D ( K' ),
		\end{equation*}
		where $ C ( \phi ):= \sup_j \sup_{ U^{ i_j } } \lV \D^2 \phi \rV $. Set $ f_j ( t ) := \int_{ M^{ i_j } } - \phi^{ i_j } \lv H_{ V^{ i_j }_t } \rv^2 + H_{ V^{ i_j }_t } \cdot \D \phi^{ i_j }\,d \lV V^{ i_j }_t \rV $, by Fatou's lemma, we have
		\begin{equation}
			\label{C+fj_liminf}
			\int_{ t_1 }^{ t_2 } \liminf_{ j \to \infty } ( C ( \phi ) D ( K' ) + f_j ( t ) )\,d t \leq \liminf_{ j \to 
				\infty } \int_{ t_1 }^{ t_2 } ( C ( \phi ) D ( K' ) + f_j ( t ) )\,d t.
		\end{equation}
		From \eqref{C+fj_liminf}, we obtain
		\begin{equation}
			\label{fj_liminf}
			\int_{ t_1 }^{ t_2 } \liminf_{ j \to \infty } f_j ( t )\,d t \leq \liminf_{ j \to \infty } \int_{ t_1 }^{ t_2 } f_j ( t )\,d t.
		\end{equation}
		Since $ ( V^{ i_j }_t , W^{ i_j }_t ) $ is a Brakke flow, the right-hand side of \eqref{fj_liminf} equals
		\begin{equation}
			\label{MC_Brake_liminf}
			\begin{split}
				&\liminf_{ j \to \infty } \Biggl( - ( \lV V^{ i_j }_t \rV + a \lV W^{ i_j }_t \rV ) ( \phi^{ i_j } ( \cdot , t ) ) \biggr|^{ t_2 }_{ t = t_1 } + \int_{ t_1 }^{ t_2 } \int_{ M^{ i_j } } \pt \phi^{ i_j }\,d ( \lV V^{ i_j }_t \rV + a \lV W^{ i_j }_t \rV ) d t \Biggr)\\
				&= - ( \lV V_t \rV + a \lV W_t \rV ) ( \phi ( \cdot , t ) ) \biggr|^{ t_2 }_{ t = t_1 } + \int_{ t_1 }^{ t_2 } \int_M \pt \phi\,d ( \lV V_t \rV + a \lV W_t \rV ) d t,
			\end{split}
		\end{equation}
		where we used the convergence of $ \lV V^{ i_j }_t \rV + a \lV W^{ i_j }_t \rV $ to $ \lV V_t \rV + a \lV W_t \rV $ for all $ t \geq 0 $. On the other hand, let the subsequence $ \{ i'_j \}_{ j \in \N } \subset \{ i_j \}_{ j \in \N } $ be such that $ ( V^{ i'_j }_t , W^{ i'_j }_t ) \rightharpoonup ( V_t , W_t ) $ as varifolds and
		\[
		\lim_{ j \to \infty } \int_{ M^{ i^{ \prime }_j } } \phi^{ i^{ \prime }_j } \lv H_{ V^{ i^{ \prime }_j }_t } \rv^2 - H_{ V^{ i^{ \prime }_j }_t } \cdot \D \phi^{ i^{ \prime }_j }\,d \lV V^{ i^{ \prime }_j }_t \rV = \liminf_{ j \to \infty } \int_{ M^{ i_j } } \phi^{ i_j } \lv H_{ V^{ i_j }_t } \rv^2 - H_{ V^{ i_j }_t } \cdot \D \phi^{ i_j }\,d \lV V^{ i_j }_t \rV.
		\]
		for almost every $ t \geq 0 $. By a layer-cake formula and Theorem \ref{compactness_for_contact_angle}, for almost every time $ t \geq 0 $, we have
		\begin{equation*}
			\begin{split}
				\int_M \phi \lv H_{ V_t } \rv^2\,d \lV V_t \rV
				&= \int^{ \infty }_0 \int_{ \{ \phi > s \} } \lv H_{ V_t } \rv^2\,d \lV V_t \rV d s
				\leq \int^{ \infty }_0 \liminf_{ j \to \infty } \int_{ \{ \phi > s \} } \lv H_{ V_t^{ i^{ \prime }_j } } \rv^2\,d \lV V_t^{ i_j } \rV d s\\
				&\leq \liminf_{ j \to \infty } \int_{ M^{ i^{ \prime }_j } } \phi \lv H_{ V_t^{ i^{ \prime }_j } } \rv^2\,d \lV V^{ i^{ \prime }_j }_t \rV = \liminf_{ j \to \infty } \int_{ M^{ i^{ \prime }_j } } \phi^{ i^{ \prime }_j } \lv H_{ V_t^{ i^{ \prime }_j } } \rv^2\,d \lV V^{ i^{ \prime }_j }_t \rV
			\end{split}
		\end{equation*}
		and
		\[
		\int_M H_{ V_t } \cdot \D \phi\,d \lV V_t \rV = \lim_{ j \to \infty } \int_{ M^{ i^{ \prime }_j } } H_{ V^{ i^{ \prime }_j }_t } \cdot \D \phi^{ i^{ \prime }_j }\,d \lV V^{ i^{ \prime }_j }_t \rV.
		\]
		The above deduce that
		\begin{equation}
			\label{-fj_liminf}
			\begin{split}
				\int_M \phi \lv H_{ V_t } \rv^2 - H_{ V_t } \cdot \D \phi\,d \lV V_t \rV &\leq \liminf_{ j \to \infty } \int_{ M^{ i^{ \prime }_j } } \phi^{ i^{ \prime }_j } \lv H_{ V^{ i^{ \prime }_j }_t } \rv^2 - H_{ V^{ i^{ \prime }_j }_t } \cdot \D \phi^{ i^{ \prime }_j }\,d \lV V^{ i^{ \prime }_j }_t \rV\\
				&= \liminf_{ j \to \infty } \int_{ M^{ i_j } } \phi^{ i_j } \lv H_{ V^{ i_j }_t } \rv^2 - H_{ V^{ i_j }_t } \cdot \D \phi^{ i_j }\,d \lV V^{ i_j }_t \rV
			\end{split}
		\end{equation}
		for almost every $ t > 0 $. Now, it follows from Lemma \ref{Basic_props} (1) and (2) that there is the set $ B \subset ( 0 , \infty ) $ such that $ t \mapsto \lV V_t \rV + a \lV W_t \rV $ is continuous at $ t \in B $ and $ B $ has full measure. By the definition of $ B $, given any $ t_i \to t $ with $ t_i , t \in B $, we have $ \lV V_{ t_i } \rV + a \lV W_{ t_i } \rV \to \lV V_t \rV + a \lV W_t \rV $. Using the convergence $ \lV V_{ t_i } \rV + a \lV W_{ t_i } \rV \to \lV V_t \rV + a \lV W_t \rV $, we deduce
		\[
		\int_M - \phi \lv H_{ V_t } \rv^2 + H_{ V_t } \cdot \D \phi\,d \lV V_t \rV \geq \limsup_{ i \to \infty } \int_M - \phi \lv H_{ V_{ t_i } } \rv^2 + H_{ V_{ t_i } } \cdot \D \phi\,d \lV V_{ t_i } \rV.
		\]
		Therefore
		\[
		t \mapsto \int_M - \phi \lv H_{ V_t } \rv^2 + H_{ V_t } \cdot \D \phi\,d \lV V_t \rV
		\]
		is upper semi-continuous on $ B \cap [ t_1 , t_2 ] $, and particularly measurable on $ [ t_1 , t_2 ] $. Thanks to \eqref{fj_liminf}, \eqref{MC_Brake_liminf}, and \eqref{-fj_liminf}, we obtain the Brakke inequality \eqref{Brakkeineq} for any $ \phi \in C^1_c ( U \times [ 0 , \infty ) ; [ 0 , \infty ) ) $ with $ \D \phi ( \cdot , t ) \in \mathcal{ T }_{ \partial M } C^0_c ( U ; \R^{ n + 1 } ) $ by approximation.
	\end{proof}
	
	\section{Elliptic regularization with contact angle}
	\label{elliptic_regularization}
	In this subsection, we adapt the elliptic regularization to the contact angle setting. The contact angle condition differs from the free-boundary setting at \cite{Edelen+2020+95+137} in which we have to consider a pair of varifolds. In \cite{ilmanen1994elliptic}, Ilmanen used the framework of rectifiable current and for general co-dimensional case. However, since we only consider hypersurfaces represented on the boundary of a domain for well definedness of contact angle varifolds, we will work with the set of finite perimeter and its boundary as in \cite{tashiro2023existence}.
	\vskip.3\baselineskip
	In the following, we will use the following notation: Throughout this section, let $ M \subset \R^{ n + 1 } $ be a compact domain with $ C^3 $ boundary. The symbol $ \e_z $ will denote the standard basis pointing the $ ( n + 2 ) $-component, that is, $ \e_z = ( 0 , \ldots , 0 , 1 ) \in \R^{ n + 1 } \times \R $. For any $ k $-dimension affine space $ T \subset \R^{ n + 1 } \times \R $, we define $ P_T $ by the projection of $ \R^{ n + 1 } \times \R $ onto $ T $. The symbols $ \mathbf{ p } $ and $ \mathbf{ q } $ will denote the projections of $ \R^{ n + 1 } \times \R $ onto its factor, that is, $ \mathbf{ p } ( x , z ) = x $ and $ \mathbf{ q } ( x , z ) = z $. Moreover, for simplicity of symbols, we set $ ( S )_i := ( M^{ \circ } \times [ 0 , \infty ) ) \cap S $ and $ ( S )_b := ( \partial M \times [ 0 , \infty ) ) \cap S $ for any set $ S \subset M \times \R $. In the same way, we define $ ( E )_i := M^{ \circ } \cap E $ and $ ( E )_b := \partial M \cap E $ for any set $ E \subset M $, without fear of confusion in notation. We fix an angle $ \theta \in ( 0 , \pi / 2 ) $ and let $ a = \cos \theta $.
	
	\subsection{Construction of approximate sequence of Brakke flow}
	\label{Constraction_approximate}
		We define the following functional for a set of finite perimeter $ S \subset M \times \R $:
		\begin{equation}
			\label{regularization_I}
			I^{ \varepsilon }_a ( S ) := \int_{ \partial^* S \cap ( M^{ \circ } \times ( - 1 , \infty ) ) } + a \int_{ \partial^* S \cap ( \partial M \times ( - 1 , \infty ) ) } \frac{ 1 }{ \varepsilon } e^{ - \frac{ z }{ \varepsilon } }\,d \mathcal{ H }^{ n + 1 } ( x , z ). 
		\end{equation}
		We will say that a minimizer $ S^{ \varepsilon } $ of $ I^{ \varepsilon }_a $ is a translating soliton with the capillarity. The name of ``translating soliton'' is based on the fact that the translation of $ \partial^* S^{ \varepsilon } $ in the $ z $-direction $ S^{ \varepsilon } - ( t / \varepsilon ) \e_z $ results in a MCF, as in the grim reaper type MCF and we will show this fact in Lemma \ref{e_Brakkeflow}. Note that $ \int \varepsilon^{ - 1 } e^{ - z / \varepsilon } \, d \mathcal{ H }^{ n + 1 } \lfloor_{ \partial S } ( x , z ) $ is the area functional for the metric $ g = e^{ - 2 z / ( ( n + 1 ) \varepsilon ) } \delta_{ \text{Eucl} } $, where $ \delta_{ \text{Eucl} } $ is the Euclidean metric. In this metric, $ \{ z = 0 \} $ is strictly convex with mean curvature pointing in the $ z $ direction.
		\vskip.3\baselineskip
		Let $ E_0 \subset M $ be a set of finite perimeter as an initial datum and we consider the minimization problem of $ I^{ \varepsilon }_a $ within the class $ C $ of sets of finite perimeter $ S' \subset M \times ( - 1 , \infty ) $ with $ \chi_S ( x , z ) = \chi_{ E_0 \times ( - 1 , 0 ] } ( x , z ) $ for $ \mathcal{ L }^{ n + 2 } $-almost everywhere $ ( x , z ) \in M \times ( - 1 , 0 ] $. From the compactness theorem of sets of locally finite perimeter, the class $ C $ is closed in the $ L^1_{loc} $ sense. Let $ \{ S^i \}_{ i = 1 }^\infty \subset C $ be a minimizing sequence for $ I_a^\varepsilon $. Since the capillary energy has the lower semi-continuous property (see \cite[Proposition 19.1]{maggi2012sets} for example) and we have local mass bounds, taking a subsequence if necessary, one can take a limit $ S^i \to S^\varepsilon \in C $ in the $ L^1_{ loc } $ sense and one can see that $ S^\varepsilon $ is a minimizer of $ I_a^\varepsilon $ among the class $ C $. From White's varifold maximum principle \cite[Theorem 1]{white2010maximum} applied to the interior varifold associated with $ \partial^* S^\varepsilon \cap ( M^{ \circ } \times [ 0 , \infty ) ) $ and the strict convexity of $ \{ z = 0 \} $, it follows that $ \mathcal{ H }^{ n + 1 } ( \partial^* S^\varepsilon \cap ( M^{ \circ } \times \{ 0 \} ) ) = 0 $. By combining this with $ \chi_{ S^\varepsilon } = \chi_{ E_0 \times ( - 1 , 0 ] } $ almost every where on $ M \times ( - 1 , 0 ] $, one may re-define $ S^\varepsilon $ so that $ \mathcal{ L }^{ n + 1 } ( S^\varepsilon_0 \setminus E_0 \cup E_0 \setminus S^\varepsilon_0 ) = 0 $, where $ S_0^\varepsilon := \{ x \in M \mid ( x , 0 ) \in S^\varepsilon \} $. To summarize, we have the following lemma (see \cite[Section 3.2]{ilmanen1994elliptic} and \cite[Section 9]{Edelen+2020+95+137} for details of the above discussion).
		\begin{lemma}
			\label{minimizer_of_I}
			Let $ E_0 \subset M $ be a set of finite perimeter. Then there exists a set of finite perimeter $ S^{ \varepsilon } \subset M \times ( - 1 , \infty ) $ such that
			\begin{enumerate}
				\item $ \mathcal{ H }^{ n + 1 } ( \partial^* S^\varepsilon \cap ( M^{ \circ } \times \{ 0 \} ) ) = 0 $ and $ \mathcal{ L }^{ n + 1 } ( S^\varepsilon_0 \setminus E_0 \cup E_0 \setminus S^\varepsilon_0 ) = 0 $, where $ S_0^\varepsilon := \{ x \in M \mid ( x , 0 ) \in S^\varepsilon \} $;
				\item for the weighted surface area of $ S^\varepsilon $ over $ M \times ( 0 , \infty ) $,
				\[
				\int_{ \partial^* S^{ \varepsilon } \cap ( M^{ \circ } \times ( 0 , \infty ) ) } + a \int_{ \partial^* S^{ \varepsilon } \cap ( \partial M \times ( 0 , \infty ) ) } \frac{ 1 }{ \varepsilon } e^{ - \frac{ z }{ \varepsilon } }\,d \mathcal{ H }^{ n + 1 } ( x , z ) \leq \mathcal{ H }^n ( ( \partial^* E_0 )_i ) + a \mathcal{ H }^n ( ( \partial^* E_0 )_b )
				\]
				holds;
				\item $ S^{ \varepsilon } $ is a minimizer of the functional $ I_a^{ \varepsilon } $ among the sets of finite perimeter $ S \subset M \times \R $ with $ \chi_S ( x , z ) = \chi_{ E_0 \times ( - 1 , 0 ] } ( x , z ) $ for $ \mathcal{ L }^{ n + 2 } $-almost everywhere $ ( x , z ) \in M \times ( - 1 , 0 ] $. In particular, $ S^\varepsilon $ is a local minimizer of $ I_a^\varepsilon $ in $ M \times ( 0 , \infty ) $.
			\end{enumerate}
	\end{lemma}
	
	For $ S^{ \varepsilon } $, calculating the first variation of $ I^{ \varepsilon }_a $ with respect to a test vector field $ X $ tangential to $ \partial M \times ( 0 , \infty ) $, we obtain the following equations and the contact angle structure can be discovered. For simplicity, an integral varifold and an integral rectifiable Radon measure are written identically.
	
	\begin{lemma}
		\label{E_L_eq}
		Let $ S^{ \varepsilon } $ be as in Lemma \ref{minimizer_of_I}. Then, $ ( \mathcal{ H }^{ n + 1 } \llcorner_{ ( \partial^* S^{ \varepsilon } )_i } , \mathcal{ H }^{ n + 1 } \llcorner_{ ( \partial^* S^{ \varepsilon } )_b } ) $ has the contact angle $ \theta $ in $ M \times ( 0 , \infty ) \subset \R^{ n + 1 } \times ( 0 , \infty ) $. Moreover, let $ H^{ \varepsilon } $ denote the mean curvature vector of $ ( \mathcal{ H }^{ n + 1 } \llcorner_{ ( \partial^* S^{ \varepsilon } )_i } , \mathcal{ H }^{ n + 1 } \llcorner_{ ( \partial^* S^{ \varepsilon } )_b } ) $ in $ M \times ( 0 , \infty ) \subset \R^{ n + 1 } \times ( 0 , \infty ) $, then we have all of the following for $ \mathcal{ H }^{ n + 1 } $-almost everywhere $ ( x , z ) \in ( \partial^* S^{ \varepsilon } )_i $:
		\begin{enumerate}
			\item $ \varepsilon H^{ \varepsilon } + P_{ T^{ \perp }_{ ( x , z ) } ( \partial^* S^{ \varepsilon } ) } ( \e_z ) = 0 $;
			\item $ \varepsilon \lv H^{ \varepsilon } \rv \leq 1 $;
			\item $ P_{ T^{ \perp }_{ ( x , z ) } ( \partial^* S^{ \varepsilon } ) } ( H^{ \varepsilon } ) = H^{ \varepsilon } $;
			\item $ \varepsilon^2 \lv H^{ \varepsilon } \rv^2 + \lv P_{ T_{ ( x , z ) } ( \partial^* S^{ \varepsilon } ) } ( \e_z ) \rv^2 = 1 $.
		\end{enumerate}
	\end{lemma}
	\begin{proof}
		Let $ X \in \mathcal{ T }_{ \partial M \times ( 0 , \infty ) } C^1_c ( \R^{ n + 1 } \times ( 0 , \infty ) ; \R^{ n + 2 } ) $, and we define the map $ \Phi^{ \delta } ( x , t ) $ by a solution of the following Cauchy problem
		\[
		\frac{ \partial }{ \partial \delta } \Phi^{ \delta } ( x , t ) = X ( \Phi^{ \delta } ( x , t ) ), \quad \Phi^0 ( x , t ) = ( x , t ), \quad \text{ for } ( x , t ) \in \R^{ n + 1 } \times ( 0 , \infty ),
		\]
		and sufficiently small $ \delta > 0 $. One deduces $ \Phi^{ \delta } ( \partial M \times ( 0 , \infty ) ) \subset \partial M \times ( 0 , \infty ) $ and $ \Phi^{ \delta } ( M \times ( 0 , \infty ) ) \subset M \times ( 0 , \infty ) $ from $ X \in \mathcal{ T }_{ \partial M \times ( 0 , \infty ) } C^1_c ( \R^{ n + 1 } \times ( 0 , \infty ) ; \R^{ n + 2 } ) $ and the uniqueness of the Cauchy problem. In particular, one obtains $ \Phi^{ \delta } ( S^{ \varepsilon } ) \subset M \times ( 0 , \infty ) $. Define $ \lv J_{ \partial^* S^{ \varepsilon } } \Phi^{ \delta } \rv $ and $ \lv J_{ \partial M \times ( 0 , \infty ) } \Phi^{ \delta } \rv $ to be the Jacobians of $ \Phi^{ \delta } $ on $ \partial^* S^{ \varepsilon } $ and $ \partial M \times ( 0 , \infty ) $, respectively. Calculating the first variation of $ I^{ \varepsilon }_a $, we obtain
		\begin{equation}
			\begin{split}
				\label{first_variation_I}
				&\frac{ d }{ d \delta } ( I^{ \varepsilon }_a ( \Phi^{ \delta } ( S^{ \varepsilon } ) ) ) \biggr|_{ \delta = 0 }\\
				&= \int_{ ( \partial^* S^{ \varepsilon } )_i } \frac{ 1 }{ \varepsilon } \frac{ d }{ d \delta } e^{ - \frac{ \mathbf{ q } ( \Phi^{ \delta } ( x , z ) ) }{ \varepsilon } } \biggr|_{ \delta = 0 } \lv J_{ \partial^* S^{ \varepsilon } } \Phi^0 \rv + \frac{ 1 }{ \varepsilon } e^{ - \frac{ \mathbf{ q } ( \Phi^0 ( x , z ) ) }{ \varepsilon } } \frac{ d }{ d \delta } \lv J_{ \partial^* S^{ \varepsilon } } \Phi^{ \delta } \rv \biggr|_{ \delta = 0 }\,d \mathcal{ H }^{ n + 1 }\\
				&\quad + a \int_{ ( \partial^* S^{ \varepsilon } )_b } \frac{ 1 }{ \varepsilon } \frac{ d }{ d \delta } e^{ - \frac{ \mathbf{ q } ( \Phi^{ \delta } ( x , z ) ) }{ \varepsilon } } \biggr|_{ \delta = 0 } \lv J_{ \partial M \times ( 0 , \infty ) } \Phi^0 \rv + \frac{ 1 }{ \varepsilon } e^{ - \frac{ \mathbf{ q } ( \Phi^0 ( x , z ) ) }{ \varepsilon } } \frac{ d }{ d \delta } \lv J_{ \partial M \times ( 0 , \infty ) } \Phi^{ \delta } \rv \biggr|_{ \delta = 0 }\,d \mathcal{ H }^{ n + 1 }\\
				&= \int_{ ( \partial^* S^{ \varepsilon } )_i } - \frac{ 1 }{ \varepsilon^2 } e^{ - \frac{ z }{ \varepsilon } } \e_z \cdot X + \frac{ 1 }{ \varepsilon } e^{ - \frac{ z }{ \varepsilon } } P_{ T_{ ( x , z ) } ( \partial^* S^{ \varepsilon } ) } : \D' X\,d \mathcal{ H }^{ n + 1 }\\
				&\quad + a \int_{ ( \partial^* S^{ \varepsilon } )_b } - \frac{ 1 }{ \varepsilon^2 } e^{ - \frac{ z }{ \varepsilon } }\e_z \cdot X + \frac{ 1 }{ \varepsilon } e^{ - \frac{ z }{ \varepsilon } } P_{ T_{ ( x , z ) } ( \partial M \times ( 0 , \infty ) ) } : \D' X\,d \mathcal{ H }^{ n + 1 }\\
				&= \int_{ ( \partial^* S^{ \varepsilon } )_i } \frac{ 1 }{ \varepsilon } \Biggl( P_{ T_{ ( x , z ) } ( \partial^* S^{ \varepsilon } ) } : \D' ( e^{ - \frac{ z }{ \varepsilon } } X ) - \frac{ 1 }{ \varepsilon }  P_{ T^{ \perp }_{ ( x , z ) } ( \partial^* S^{ \varepsilon } ) } ( \e_z ) \cdot ( e^{ - \frac{ z }{ \varepsilon } } X ) \Biggr)\,d \mathcal{ H }^{ n + 1 }\\
				&\quad + a \int_{ ( \partial^* S^{ \varepsilon } )_b } \frac{ 1 }{ \varepsilon } \Biggl( P_{ T_{ ( x , z ) } ( \partial^* S^{ \varepsilon } ) } : \D' ( e^{ - \frac{ z }{ \varepsilon } } X ) - \frac{ 1 }{ \varepsilon }  P_{ T^{ \perp }_{ ( x , z ) } ( \partial M \times ( 0 , \infty ) ) } ( \e_z ) \cdot ( e^{ - \frac{ z }{ \varepsilon } } X ) \Biggr)\,d \mathcal{ H }^{ n + 1 }.
			\end{split}
		\end{equation}
		Since $ S^{ \varepsilon } $ is the minimizer of $ I^{ \varepsilon }_a $ and $ X \in \mathcal{ T }_{ \partial M \times ( 0 , \infty ) } C^1_c ( \R^{ n + 1 } \times ( 0 , \infty ) ; \R^{ n + 2 } ) $, it follows from replacing $ e^{ - z / \varepsilon } X $ to $ X $ in \eqref{first_variation_I} that
		\begin{equation}
			\label{translative_MC}
			\begin{split}
				\int_{ ( \partial^* S^{ \varepsilon } )_i } &P_{ T_{ ( x , z ) } ( \partial^* S^{ \varepsilon } ) } : \D' X\,d \mathcal{ H }^{ n + 1 } + a \int_{ ( \partial^* S^{ \varepsilon } )_b } P_{ T_{ ( x , z ) } ( \partial M \times ( 0 , \infty ) ) } : \D' X\,d \mathcal{ H }^{ n + 1 }\\
				&= \int_{ ( \partial^* S^{ \varepsilon } )_i } \frac{ 1 }{ \varepsilon } P_{ T^{ \perp }_{ ( x , z ) } ( \partial^* S^{ \varepsilon } ) } ( \e_z ) \cdot X\,d \mathcal{ H }^{ n + 1 },
			\end{split}
		\end{equation}
		which implies that $ ( \mathcal{ H }^{ n + 1 } \llcorner_{ ( \partial^* S^{ \varepsilon } )_i } , \mathcal{ H }^{ n + 1 } \llcorner_{ ( \partial^* S^{ \varepsilon } )_b } ) $ has the contact angle $ \theta $ in $ M \times ( 0 , \infty ) $ and its mean curvature $ H^{ \varepsilon } $ equals $ - P_{ T^{ \perp }_{ ( x , z ) } ( \partial^* S^{ \varepsilon } ) } ( \e_z ) / \varepsilon $. Equations (1)-(4) follow from the obtained mean curvature formula.
	\end{proof}
	
	Here, we provide an overview of the changes in the argument in
	\cite[Section 4]{ilmanen1994elliptic} under the contact angle condition. For the detailed discussion, therefore, see \cite[Section 4]{ilmanen1994elliptic}. For all $ \xi \in C^1_c ( ( 0 , \infty ) ) $, note that $ X ( x , z ) = \xi ( z ) \e_z $ is of $ \mathcal{ T }_{ \partial M \times ( 0 , \infty ) } C^1_c ( \R^{ n + 1 } \times ( 0 , \infty ) ; \R^{ n + 2 } ) $. Plugging this $ X $ into the third line of \eqref{first_variation_I}, we have
	\begin{equation}
		\label{time_slice_1}
		\begin{split}
			0 = &\int_{ ( \partial^* S^{ \varepsilon } )_i } - \frac{ 1 }{ \varepsilon } e^{ - \frac{ z }{ \varepsilon } } \xi + e^{ - \frac{ z }{ \varepsilon } } \e_z \cdot P_{ T_{ ( x , z ) } ( \partial^* S^{ \varepsilon } ) } ( \e_z ) \partial_z \xi\,d \mathcal{ H }^{ n + 1 }\\
			&+ a \int_{ ( \partial^* S^{ \varepsilon } )_b } - \frac{ 1 }{ \varepsilon } e^{ - \frac{ z }{ \varepsilon } } \xi + e^{ - \frac{ z }{ \varepsilon } } \e_z \cdot P_{ T_{ ( x , z ) } ( \partial M \times ( 0 , \infty ) ) } ( \e_z ) \partial_z \xi\,d \mathcal{ H }^{ n + 1 }.
		\end{split}
	\end{equation}
	Let $ \xi : ( 0 , \infty ) \to \R $ be Lipschitz with $ \supp \xi \subset \subset ( 0 , \infty ) $. Consider the approximation of $ \xi $ by $ \xi^i $ such that
	\[
	\xi^i \leq \xi , \quad \xi^i \to \xi \text{ uniformly}, \quad \partial_z \xi^i \rightharpoonup \partial_z \xi \text{ weakly-$ {}^* $ in } L^{ \infty } ( 0 , \infty ),
	\]
	\eqref{time_slice_1} is also true for the Lipschitz function $ \xi $ (such an approximation can be taken by general measure theory, see \cite[Theorem 6.11]{evans2015measure} for example). Replacing $ \xi $ to $ e^{ z / \varepsilon } \xi $ in \eqref{time_slice_1}, since $ \varepsilon H^{ \varepsilon } = - P_{ T^{ \perp }_{ ( x , z ) } ( \partial^* S^{ \varepsilon } ) } ( \e_z ) $ for $ \mathcal{ H }^{ n + 1 } \llcorner_{ ( \partial^* S^{ \varepsilon } )_i } $-almost everywhere by Lemma \ref{E_L_eq} (1) and $ \e_z $ is tangential to $ \partial M \times ( 0 , \infty ) $, we have
	\begin{equation}
		\label{times_slice_2}
		\begin{split}
			0 = &\int_{ ( \partial^* S^{ \varepsilon } )_i } - \varepsilon \xi \lv H^{ \varepsilon } \rv^2 + \lv P_{ T_{ ( x , z ) } ( \partial^* S^{ \varepsilon } ) } ( \e_z ) \rv^2 \partial_z \xi\,d \mathcal{ H }^{ n + 1 }\\
			&+ a \int_{ ( \partial^* S^{ \varepsilon } )_b } \partial_z \xi\,d \mathcal{ H }^{ n + 1 }.
		\end{split}
	\end{equation}
	Let $ 0 < z_1 < z_2 $ be arbitrary and we define a Lipschitz map $ \xi $ for sufficiently small $ \delta > 0 $ by
	\[
	\xi_{ \delta } ( z ) := \begin{cases}
		0 &( z \in [ 0 , z_1 ] )\\
		1 &( z \in [ z_1 + \delta , z_2 ] )\\
		0 &( z \in [ z_2 + \delta , \infty ) )
	\end{cases}
	\]
	and linearly interpolated between. Plugging this $ \xi_{ \delta } $ into \eqref{times_slice_2} leads to
	\begin{equation*}
		\begin{split}
			&\int_{ ( \partial^* S^{ \varepsilon } )_i } \varepsilon \xi_{ \delta } \lv H^{ \varepsilon } \rv^2\,d \mathcal{ H }^{ n + 1 }\\
			&= \Biggl( \frac{ 1 }{ \delta } \int_{ ( M^{ \circ } \times ( z , z + \delta ) ) \cap \partial^* S^{ \varepsilon } } \lv P_{ T_{ ( x , z ) } ( \partial^* S^{ \varepsilon } ) } ( \e_z ) \rv^2\,d \mathcal{ H }^{ n + 1 } + \frac{ 1 }{ \delta } \int_{ ( \partial M \times ( z , z + \delta ) ) \cap \partial^* S^{ \varepsilon } } a\,d \mathcal{ H }^{ n + 1 } \Biggr) \Biggr\rv^{ z_2 }_{ z = z_1 }.
		\end{split}
	\end{equation*}
	This implies that, for any fixed $ \delta > 0 $, the function
	\[
	f_{ \delta } ( z ) := \frac{ 1 }{ \delta } \int_{ ( M^{ \circ } \times ( z , z + \delta ) ) \cap \partial^* S^{ \varepsilon } } \lv P_{ T_{ ( x , z ) } ( \partial^* S^{ \varepsilon } ) } ( \e_z ) \rv^2\,d \mathcal{ H }^{ n + 1 } + \frac{ 1 }{ \delta } \int_{ ( \partial M \times ( z , z + \delta ) ) \cap \partial^* S^{ \varepsilon } } a\,d \mathcal{ H }^{ n + 1 }
	\]
	is decreasing in $ z \geq 0 $. Next, we define a Lipschitz map $ \xi_{ L , \delta } $ for $ \delta > 0 $ and $ L > 0 $ by
	\[
	\xi_{ L , \delta } ( z ) := \begin{cases}
		0 &( z \in [ 0 , z_1 ) )\\
		1 &( z = z_1 + \delta )\\
		0 &( z \in [ z_1 + \delta + L , \infty ) )
	\end{cases}
	\]
	and linearly interpolated between. Plugging this $ \xi_{ L , \delta } $ into \eqref{time_slice_1}, we compute
	\begin{equation*}
		\begin{split}
			&\int_{ ( \partial^* S^{ \varepsilon } )_i } + a \int_{ ( \partial^* S^{ \varepsilon } )_b } \frac{ 1 }{ \varepsilon } e^{ - \frac{ z }{ \varepsilon } } \xi_{ L , \delta }\,d \mathcal{ H }^{ n + 1 }\\
			&= \frac{ 1 }{ \delta } \int_{ ( M^{ \circ } \times ( z_1 , z_1 + \delta ) ) \cap \partial^* S^{ \varepsilon } } + \frac{ a }{ \delta } \int_{ ( \partial M \times ( z_1 , z_1 + \delta ) ) \cap \partial^* S^{ \varepsilon } } e^{ - \frac{ z }{ \varepsilon } } \lv P_{ T_{ ( x , z ) } ( \partial^* S^{ \varepsilon } ) } ( \e_z ) \rv^2\,d \mathcal{ H }^{ n + 1 }\\
			&- \frac{ 1 }{ L } \int_{ ( M^{ \circ } \times ( z_1 + \delta , z_1 + \delta + L ) ) } - \frac{ a }{ L } \int_{ ( \partial M \times ( z_1 + \delta, z_1 + \delta + L ) ) \cap \partial^* S^{ \varepsilon } } e^{ - \frac{ z }{ \varepsilon } } \lv P_{ T_{ ( x , z ) } ( \partial^* S^{ \varepsilon } ) } ( \e_z ) \rv^2\,d \mathcal{ H }^{ n + 1 }.
		\end{split}
	\end{equation*}
	The sum of the bottom two terms is bounded by $ ( \varepsilon I^{ \varepsilon }_a ( S^{ \varepsilon } ) ) / L $, and these terms vanish as $ L \to \infty $. Thus, by taking $ L \to \infty $, and then $ \delta \to + 0 $, we obtain
	\[
	\lim_{ \delta \to + 0 } \frac{ 1 }{ \delta } \Biggl( \int_{ ( M^{ \circ } \times ( z_1 , z_1 + \delta ) ) \cap \partial^* S^{ \varepsilon } } + a \int_{ ( \partial M \times ( z_1 , z_1 + \delta ) ) \cap \partial^* S^{ \varepsilon } } e^{ - \frac{ z }{ \varepsilon } } \lv P_{ T_{ ( x , z ) } ( \partial^* S^{ \varepsilon } ) } ( \e_z ) \rv^2\,d \mathcal{ H }^{ n + 1 } \Biggr) \leq I^{ \varepsilon }_a ( S^{ \varepsilon } ).
	\]
	We approximate the measurable set $ A \subset [ 0 , \infty ) $ to many small intervals and use the inequality above. By taking the limit of the width of the intervals to $ 0 $ and the monotonicity of $ f_{ \delta } $ and Lemma \ref{minimizer_of_I} (1), we obtain
	\[
	\frac{ 1 }{ \mathcal{ L }^1 ( A ) } \Biggl( \int_{ ( M^{ \circ } \times A ) \cap \partial^* S^{ \varepsilon } } + a \int_{ ( \partial M \times A ) \cap \partial^* S^{ \varepsilon } } \lv P_{ T_{ ( x , z ) } ( \partial^* S^{ \varepsilon } ) } ( \e_z ) \rv^2\,d \mathcal{ H }^{ n + 1 } \Biggr) \leq I^{ \varepsilon }_a ( S^{ \varepsilon } )
	\]
	for any measurable set $ A \subset [ 0 , \infty ) $. Therefore, by the Lebesgue differentiation theorem and the co-area formula (see \cite[Theorem 13.1]{maggi2012sets}, for example), we obtain the following lemma.
	
	\begin{lemma}
		For almost every $ 0 \leq z_1 < z_2 < \infty $,
		\begin{equation}
			\label{time_slice_3}
			\begin{split}
				&\int_{ ( \partial^* S^{ \varepsilon }_z )_i } \lv P_{ T_{ ( x , z ) } ( \partial^* S^{ \varepsilon } ) } ( \e_z ) \rv\,d \mathcal{ H }^n ( x ) \biggr|^{ z_2 }_{ z = z_1 }
				+ a \int_{ ( \partial^* S^{ \varepsilon }_z )_b } d \mathcal{ H }^n ( x )
				\biggr|^{ z_2 }_{ z = z_1 }\\
				&= - \int_{ ( M^{ \circ } \times ( z_1 , z_2 ) ) \cap \partial^* S^{ \varepsilon } } \varepsilon \lv H^{ \varepsilon } \rv^2\,d \mathcal{ H }^{ n + 1 } ( x , z ).
			\end{split}
		\end{equation}
		In particular, by Lemma \ref{minimizer_of_I} (2),
		\begin{equation}
			\label{L2MC_estimate}
			\int_{ ( \partial^* S^{ \varepsilon } )_i } \varepsilon \lv H^{ \varepsilon } \rv^2\,d \mathcal{ H }^{ n + 1 } ( x , z ) \leq \mathcal{ H }^n ( ( \partial^* E_0 )_i ) + a \mathcal{ H }^n ( ( \partial^* E_0 )_b ),
		\end{equation}
		\begin{equation}
			\label{S_mass_estimate}
			\begin{split}
				\sup_{ z > 0 } \Biggl( \int_{ ( \partial^* S^{ \varepsilon }_z )_i } \lv P_{ T_{ ( x , z ) } ( \partial^* S^{ \varepsilon } ) } ( \e_z ) \rv\,d \mathcal{ H }^n + a \int_{ ( \partial^* S^{ \varepsilon }_z )_b } d \mathcal{ H }^n \Biggr) \leq \mathcal{ H }^n ( ( \partial^* E_0 )_i ) + a \mathcal{ H }^n ( ( \partial^* E_0 )_b ).
			\end{split}
		\end{equation}
	\end{lemma}
	
	Now consider $ E^{ \varepsilon } := \kappa_{ \varepsilon } ( S^{ \varepsilon } ) $ in which $ S^{ \varepsilon } $ is shrunk by the map $ \kappa_{ \varepsilon } ( x , z ) = ( x , \varepsilon z ) $ for $ ( x , z ) \in M \times ( 0 , \infty ) $. By the same calculation as in \cite[Section 5.1 and 5.3]{ilmanen1994elliptic}, one can obtain the following lemma for the mass of $ \partial^* S^{ \varepsilon } $ and $ \partial^* E^{ \varepsilon } $.
	
	\begin{lemma}
		\label{boundedness_SandE}
		For any open interval $ I = ( z_1 , z_2 ) \subset [ 0 , \infty ) $, we obtain
		\begin{equation}
			\label{themassbddS}
			\begin{split}
				( \mathcal{ H }^{ n + 1 } \llcorner_{ ( \partial^* S^{ \varepsilon } )_i } + a \mathcal{ H }^{ n + 1 } \llcorner_{ ( \partial^* S^{ \varepsilon } )_b } ) ( M \times I ) \leq \bigl( \mathcal{ L }^1 ( I ) + \varepsilon \bigr) ( \mathcal{ H }^n ( ( \partial^* E_0 )_i ) + a \mathcal{ H }^n ( ( \partial^* E_0 )_b ) ),
			\end{split}
		\end{equation}
		\begin{equation}
			\label{themassbddE}
			\begin{split}
				( \mathcal{ H }^{ n + 1 } \llcorner_{ ( \partial^* E^{ \varepsilon } )_i } &+ a \mathcal{ H }^{ n + 1 } \llcorner_{ ( \partial^* E^{ \varepsilon } )_b } ) ( M \times I )\\
				&\leq \bigl( \mathcal{ L }^1 ( I ) + \varepsilon^2 + ( \mathcal{ L }^1 ( I ) + \varepsilon^2 )^{ \frac 1 2 } \bigr) ( \mathcal{ H }^n ( ( \partial^* E_0 )_i ) + a \mathcal{ H }^n ( ( \partial^* E_0 )_b ) ).
			\end{split}
		\end{equation}
		In particular, the result holds for any $ \mathcal{ L }^1 $-measurable set $ I \subset [ 0 , \infty ) $ by approximation and Lemma \ref{minimizer_of_I} (1).
	\end{lemma}
	
	For this $ S^{ \varepsilon } $, we define the following notations;
	\begin{equation}
		\label{notations}
		\begin{split}
			\sigma_{ - t / \varepsilon } ( x , z ) := \Biggl( x , z - \frac{ t }{ \varepsilon } \Biggr),& \quad S^{ \varepsilon } ( t ) := \sigma_{ - t / \varepsilon } ( S^{ \varepsilon} ),\\
			\mu_t^{ \varepsilon } := \mathcal{ H }^{ n + 1 } \llcorner_{ \partial^* S^{ \varepsilon } ( t ) \cap ( M^{ \circ } \times ( - t / \varepsilon , \infty ) ) },& \quad \nu_t^{ \varepsilon } := \mathcal{ H }^{ n + 1 } \llcorner_{ \partial^* S^{ \varepsilon } ( t ) \cap ( \partial M \times ( - t / \varepsilon , \infty ) ) }.
		\end{split}
	\end{equation}
	
	The Euler--Lagrange equation in Lemma \ref{E_L_eq} implies that the above translated surface measure $ \mu^{ \varepsilon }_t $ must be a MCF in the interior region. Moreover, we show that the pair $ ( \mu^{ \varepsilon }_t , \nu^{ \varepsilon }_t ) $ is a Brakke flow as defined in Definition \ref{def_of_Brakke} up to the boundary.
	
	\begin{lemma}
		\label{e_Brakkeflow}
		The pair of varifolds $ ( \mu^{ \varepsilon }_t , \nu^{ \varepsilon }_t ) $ is a Brakke flow in the set
		\[
		W^{ \varepsilon } := \Biggl\{ ( x , z , t ) \in ( M \times \R ) \times [ 0 , \infty ) : z > - \frac{ t }{ \varepsilon } \Biggr\},
		\]
		and has the contact angle $ \theta $ in $ M \times ( - t / \varepsilon , \infty ) \subset \R^{ n + 1 } \times ( - t / \varepsilon , \infty ) $ for each $ t > 0 $.
	\end{lemma}
	\begin{proof}
		By the definition of $ \mu^{ \varepsilon }_t $ and $ \nu^{ \varepsilon }_t $, these measures are integer and have the contact angle $ \theta $ in $ M \times ( - t / \varepsilon , \infty ) $. Let $ t \geq 0 $ fix and let $ \phi \in C^2_c ( \R^{ n + 1 } \times \R \times ( 0 , \infty ) ) $ be such that $ \supp \phi \subset W^{ \varepsilon } $, $ \D' \phi ( \cdot , \cdot , s ) \in \mathcal{ T }_{ \partial M \times ( - t / \varepsilon ) } C^1_c ( \R^{ n + 1 } \times ( - t / \varepsilon ) ; \R^{ n + 2 } ) $, and $ \phi \geq 0 $. By Lemma \ref{E_L_eq} (1) and (3), for $ \mu^{ \varepsilon }_t $-almost everywhere $ ( x , z ) \in M \times ( - t / \varepsilon , \infty ) $, we have
		\[
		P_{ T^{ \perp }_{ ( x , z ) } ( \partial^* S^{ \varepsilon } ( t ) ) } ( \tilde{ H }^{ \varepsilon } ( x , z ) ) = \tilde{ H }^{ \varepsilon } ( x , z ) = \frac{ 1 }{ \varepsilon } P_{ T^{ \perp }_{ ( x , z ) } ( \partial^* S^{ \varepsilon } ( t ) ) } ( \e_z ),
		\]
		where $ \tilde{ H } $ is the mean curvature vector of $ \partial^* S^{ \varepsilon } ( t ) $. Hence, plugging
		\[
		X ( x , z ) = - \varepsilon^{ - 1 } \phi ( x , z - t / \varepsilon , t ) \e_z
		\]
		into \eqref{first_variation_I}, we obtain
		\begin{equation*}
			\begin{split}
				&\frac{ d }{ d t } ( \mu^{ \varepsilon }_t + a \nu^{ \varepsilon }_t ) ( \phi )
				= \frac{ d }{ d t } \int_{ M \times \R } \phi ( x , z - t / \varepsilon , t )\,d ( \mathcal{ H }^{ n + 1 } \llcorner_{ ( \partial^* S^{ \varepsilon } )_i } + a \mathcal{ H }^{ n + 1 } \llcorner_{ ( \partial^* S^{ \varepsilon } )_b } )\\
				&= \int_{ ( \partial^* S^{ \varepsilon } )_i } + a \int_{ ( \partial^* S^{ \varepsilon } )_b } \Biggl( - \frac{ 1 }{ \varepsilon } \e_z \cdot \D' \phi ( x , z - t / \varepsilon , t ) + \partial_t \phi ( x , z - t / \varepsilon , t ) \Biggr)\,d \mathcal{ H }^{ n + 1 }\\
				&= \int_{ ( \partial^* S^{ \varepsilon } )_i } + a \int_{ ( \partial^* S^{ \varepsilon } )_b } \Biggl( - \frac{ 1 }{ \varepsilon } P_{ T_{ ( x , z ) } ( \partial^* S^{ \varepsilon } ) } ( \e_z ) \cdot \D' \phi ( x , z - t / \varepsilon , t ) \Biggr)\,d \mathcal{ H }^{ n + 1 }\\
				&+ \int_{ ( \partial^* S^{ \varepsilon } )_i } - \frac{ 1 }{ \varepsilon } P_{ T^{ \perp }_{ ( x , z ) } ( \partial^* S^{ \varepsilon } ) } ( \e_z ) \cdot \D' \phi ( x , z - t / \varepsilon , t )\,d \mathcal{ H }^{ n + 1 }\\
				&+ \int_{ ( \partial^* S^{ \varepsilon } )_i } + a \int_{ ( \partial^* S^{ \varepsilon } )_b } \pt \phi ( x , z - t / \varepsilon , t )\,d \mathcal{ H }^{ n + 1 }\\
				&= \int_{ M \times \R } \Bigl( \D' \phi - \phi H^{ \varepsilon } \Bigr) \cdot H^{ \varepsilon }\,d \mu^{ \varepsilon }_t + \pt \phi\,d ( \mu^{ \varepsilon }_t + a \nu^{ \varepsilon }_t ).
			\end{split}
		\end{equation*}
		This implies that $ ( \mu^{ \varepsilon }_t , \nu^{ \varepsilon }_t ) $ is a Brakke flow with the contact angle $ \theta $.
	\end{proof}
	
	\subsection{Estimates on first variations and existence of Brakke flow}
	\label{Estimates on first variations and existence of Brakke flow}
	In the previous section, we introduced the translating soliton $ \partial^* S^{ \varepsilon } - t \e_z / \varepsilon $ and proved that it is a Brakke flow. As $ \varepsilon \to 0 $, this $ \varepsilon $-soliton intuitively extends to a $ z $-invariant cylinder $ \partial^* E_t \times \R $ with initial condition $ \partial^* E_0 \times ( 0 , \infty ) $. If the compactness theorem of the Brakke flow works for the construction via the elliptic regularization with capillarity, one can deduce that a cylinder $ \partial^* E_t \times \R $ is a Brakke flow and the slicing argument to this cylinder may give us the desired Brakke flow $ \partial^* E_t $. In this subsection, we estimate the first variations of both $ \mathcal{ H }^{ n + 1 } \llcorner_{ ( \partial^* S^{ \varepsilon } )_i } $ and $ \mathcal{ H }^{ n + 1 } \llcorner_{ ( \partial^* S^{ \varepsilon } )_b } $ individually in order to justify the above discussion. These estimates are crucial for applying the compactness theorem and proving the existence of a Brakke flow with the initial condition $ \partial^* E_0 $.
	\vskip.3\baselineskip
	To estimate the first variations, it is essential that $ \mathcal{ H }^{ n + 1 } \llcorner_{ ( \partial^* S^{ \varepsilon } )_i } $ and $ \mathcal{ H }^{ n + 1 } \llcorner_{ ( \partial^* S^{ \varepsilon } )_b } $ are locally finite, respectively. This requirement follows the fact that $ S^{ \varepsilon } $ is a minimizer. Indeed, De Philippis and Maggi \cite{philippis2015regularity} established the boundary regularity theorem for the class of almost-minimizers. We here present a less general version of it to extent that it can be used in this paper.
	
	\begin{theorem}
		\label{reg_for_S}
		Let $ A \subset \R^{ n + 1 } \times ( 0 , \infty ) $ be an open set, $ r_0 \in ( 0 , \infty ) $. Let a weight function $ \Phi : \R^{ n + 1 } \times [ 0 , \infty ) \to \R $ satisfy, for all $ ( x_1 , z_1 ) $, $ ( x_2 , z_2 ) \in A \cap ( M \times ( 0 , \infty ) ) $,
		\[
		\lambda^{ - 1 } \leq \Phi ( x_1 , z_1 ) \leq \lambda, \quad \lv \Phi ( x_1 , z_1 ) - \Phi ( x_2 , z_2 ) \rv \leq l \lv ( x_1 , z_1 ) - ( x_2 , z_2 ) \rv
		\]
		for some $ \lambda \geq 1 $ and $ l \geq 0 $. Let $ S $ be a set of finite perimeter satisfying
		\begin{equation*}
			\begin{split}
				\int_{ ( \partial^* S )_i \cap W } \Phi\,d \mathcal{ H }^{ n + 1 } + a \int_{ ( \partial^* S )_b \cap W } \Phi\,d \mathcal{ H }^{ n + 1 } \leq \int_{ ( \partial^* F )_i \cap W } \Phi\,d \mathcal{ H }^{ n + 1 } + a \int_{ ( \partial^* F )_b \cap W } \Phi\,d \mathcal{ H }^{ n + 1 },
			\end{split}
		\end{equation*}
		whenever $ F \subset M \times ( 0 , \infty ) $, $ ( ( F \setminus S ) \cup ( S \setminus F ) ) \subset \subset W $, and $ W \subset \subset A $ is open with $ \mathrm{ diam } W < 2 r_0 $. Then there is an open set $ A^{ \prime } \subset A $ with $ A \cap ( \partial M \times ( 0 , \infty ) ) = A^{ \prime } \cap ( \partial M \times ( 0 , \infty ) ) $ such that $ S $ is equivalent to an open set in $ A^{ \prime } $ and $ ( \partial S )_b = \partial S \cap ( \partial M \times ( 0 , \infty ) ) $ is a set of locally finite perimeter in $ A^{ \prime } \cap ( \partial M \times ( 0 , \infty ) ) $ (equivalently in $ A \cap ( \partial M \times ( 0 , \infty ) ) $). Moreover, let $ \Sigma = \overline{ ( \partial S )_b } $, then we have
		\[
		\partial_{ \partial M \times ( 0 , \infty ) } ( ( \partial S )_b ) \cap A = \partial_{ \partial M \times ( 0 , \infty ) } ( ( \partial S )_b ) \cap A^{ \prime } = ( \Sigma )_b,
		\]
		where $ \partial_{ \partial M \times ( 0 , \infty ) } ( \cdot ) $ represents the topological boundary relative to $ \partial M \times ( 0 , \infty ) $, and there exists a relatively closed set $ H \subset ( \Sigma )_b $ such that $ \mathcal{ H }^n ( H ) = 0 $ and for every $ ( x , z ) \in ( \Sigma )_b \setminus H $, $ \Sigma $ is a $ C^{ 1 , 1 / 2 } $ manifold with boundary in a neighborhood of $ ( x , z ) $ for which $ \nu_S \cdot \nu_{ \partial M \times ( 0 , \infty ) } = a $ for all $ x \in ( \Sigma )_b \setminus H $.
	\end{theorem}
	
	Let $ 0 < z_1 < z_2 $, $ A = \R^{ n + 1 } \times ( z_1 , z_2 ) $, and let $ 0 < r_0 \leq z_1 $ be a sufficiently small number depending only on the boundary $ \partial M \times ( 0 , \infty )$. Since the weight function $ e^{ - z / \varepsilon } $ satisfies
	\[
	e^{ - \frac{ z_2 }{ \varepsilon } } \leq e^{ - \frac{ z }{ \varepsilon } } \leq e^{ - \frac{ z_1 }{ \varepsilon } } < e^{ \frac{ z_2 }{ \varepsilon } }, \quad \mathrm{ Lip } \Bigl( e^{ - \frac{ z }{ \varepsilon } } \Bigr) = \frac{ 1 }{ \varepsilon } \quad \text{ in } A \cap ( M \times ( 0 , \infty ) )
	\]
	and we take $ S^{ \varepsilon } $ as a minimizer of $ I^{ \varepsilon }_a $, one can apply the above regularity theorem to $ S^{ \varepsilon } $ for each $ \varepsilon $. As a corollary, the two rectifiable Radon measure $ \mathcal{ H }^{ n + 1 } \llcorner_{ ( \partial^* S^{ \varepsilon } )_i } $ and $ \mathcal{ H }^{ n + 1 } \llcorner_{ ( \partial^* S^{ \varepsilon } )_b } $ are locally bounded, and one can see that $ \mathcal{ H }^{ n + 1 } \llcorner_{ ( \partial^* S^{ \varepsilon } )_i } $ meets the boundary at angle $ \theta $ in a weak sense. In particular, the following holds from Proposition \ref{constancy}, Theorem \ref{reg_for_S} and $ \mathcal{ H }^{ n + 1 } \llcorner_{ ( \partial^* S^{ \varepsilon } )_i } ( \partial M \times ( 0 , \infty ) ) = 0 $.
	\begin{corollary}
		For each $ \varepsilon $, the two rectifiable Radon measure $ \mathcal{ H }^{ n + 1 } \llcorner_{ ( \partial^* S^{ \varepsilon } )_i } $ and $ \mathcal{ H }^{ n + 1 } \llcorner_{ ( \partial^* S^{ \varepsilon } )_b } $ are locally bounded. In particular, there exist the boundary Radon measures $ \sigma_i $ and $ \sigma_b $ such that the following holds for any $ X \in C^1_c ( \R^{ n + 1 } \times ( 0 , \infty ) : \R^{ n + 2 } ) $:
		\begin{equation}
			\label{weak_reg}
			\begin{split}
				\delta \mathcal{ H }^{ n + 1 } \llcorner_{ ( \partial^* S^{ \varepsilon } )_i } ( X ) &= - \int_{ ( \partial^* S^{ \varepsilon } )_i } X \cdot H^{ \varepsilon }\,d \mathcal{ H }^{ n + 1 } + \int_{ \partial M \times ( 0 , \infty ) } X \cdot \eta_{ \sigma_i }\,d \sigma_i,\\
				\delta \mathcal{ H }^{ n + 1 } \llcorner_{ ( \partial^* S^{ \varepsilon } )_b } ( X ) &= - \int_{ ( \partial^* S^{ \varepsilon } )_b } X \cdot H_{ \partial M \times ( 0 , \infty ) }\,d \mathcal{ H }^{ n + 1 } + \int_{ \partial M \times ( 0 , \infty ) } X \cdot \eta_{ \sigma_b }\,d \sigma_b,
			\end{split}
		\end{equation}
		where $ H_{ \partial M \times ( 0 , \infty ) } $ is the mean curvature vector of $ \partial M \times ( 0 , \infty ) $ and $ \eta_{ \sigma_i } $, $ \eta_{ \sigma_b } $ are the direction vectors of $ \sigma_i $, $ \sigma_b $ respectively. Moreover, $ \eta_{ \sigma_i } $ and $ \eta_{ \sigma_b } $ satisfy
		\begin{equation}
			\label{meets_angle}
			\eta_{ \sigma_i } \cdot \nu_{ \partial M \times ( 0 , \infty ) } = \sqrt{ 1 - a^2 } \text{ for $ \sigma_i $-a.e., } \quad \eta_{ \sigma_b } \cdot \nu_{ \partial M \times ( 0 , \infty ) } = 0 \text{ for $ \sigma_b $-a.e. }
		\end{equation}
	\end{corollary}
	
	However, the above corollary does not provide uniform estimates for the first variations, which are necessary to apply the compactness theorem. Specifically, the estimates proved in \cite[Section 2.7]{philippis2015regularity} depend on the weight function, and hence on $ \varepsilon $. To obtain the $ \varepsilon $-independent estimates, we prove the following bounds in terms of the curvature, which is based on the arguments in \cite[Theorem 1.1]{de2021rectifiability} and \cite[Theorem 3.3]{de2022existence}.
	
	\begin{lemma}
		For any $ 0 < z_1 < z_2 $, there exists a constant $ C = C ( a , M , z_1 , z_2 ) > 0 $ such that
		\begin{equation}
			\label{firstV_esti}
			\begin{split}
				&\lV \delta \mathcal{ H }^{ n + 1 } \llcorner_{ ( \partial^* S^{ \varepsilon } )_i } \rV \Biggl( M \times \Biggl( \frac{ z_1 }{ 2 } ,  \frac{ z_2 }{ 2 } \Biggr) \Biggr) \leq C \int_{ ( \partial^* S^{ \varepsilon } )_i \cap ( M \times ( z_1 , z_2 ) ) } ( 1 + \lv H^{ \varepsilon } \rv )\,d \mathcal{ H }^{ n + 1 },\\
				&\lV \delta \mathcal{ H }^{ n + 1 } \llcorner_{ ( \partial^* S^{ \varepsilon } )_b } \rV \Biggl( M \times \Biggl( \frac{ z_1 }{ 2 } ,  \frac{ z_2 }{ 2 } \Biggr) \Biggr) \leq C \int_{ \partial^* S^{ \varepsilon } \cap ( M \times ( z_1 , z_2 ) ) } ( 1 + \lv H^{ \varepsilon } \rv )\,d \mathcal{ H }^{ n + 1 }.
			\end{split}
		\end{equation}
	\end{lemma}
	\begin{proof}
		We here estimate the boundary measure $ \sigma_i $. Let $ ( x , z ) \in \partial M \times ( z_1 , z_2 ) $. We fix a smooth radial cutoff function $ \varphi \in C^{ \infty }_c ( \R^{ n + 2 } ) $ centered at $ ( x , z ) $ such that
		\[
		\varphi = 1 \text{ on } B_{ 1 / 2 } ( x , z ),\ \varphi = 0 \text{ outside } B_1 ( x , z ),\ \lv \nabla^{ \prime } \varphi \rv \leq 3.
		\]
		Let $ \delta > 0 $ be a number such that the signed distance $ d = d_{ \partial M \times ( 0 , \infty ) } $ from $ \partial M \times ( 0 , \infty ) $ is $ C^2 $ in $ \{ ( x , z ) \in M \times ( 0 , \infty ) \mid \lv d ( x , z ) \rv \leq \delta \} $. We note that such a number exists by \cite[Lemma 14.16]{gilbarg2001elliptic}. Let $ 0 < r \leq \delta $ be arbitrary. Then, by \eqref{weak_reg} and \eqref{meets_angle}, we obtain
		\begin{equation*}
			\begin{split}
				&\sigma_i ( B_{ r / 2 } ( x , z ) ) \leq \int_{ \partial M \times ( 0 , \infty ) } \varphi \Biggl( \frac{ \lv ( y , s ) \rv }{ r } \Biggr)\,d \sigma_i ( y , s )\\
				&= - \frac{ 1 }{ \sqrt{ 1 - a^2 } } \int_{ \partial M \times ( 0 , \infty ) } \varphi \Biggl( \frac{ \lv ( y , s ) \rv }{ r } \Biggr) \nabla^{ \prime } d ( y , s ) \cdot \eta_{ \sigma_i }\,d \sigma_i ( y , s )\\
				&= - \frac{ 1 }{ \sqrt{ 1 - a^2 } } \int_{ ( \partial^* S^{ \varepsilon } )_i } \mathrm{ div }_{ T_{ ( y , s ) } ( \partial^* S^{ \varepsilon } )_i } \Biggl( \varphi \Biggl( \frac{ \lv ( y , s ) \rv }{ r }  \Biggr) \nabla^{ \prime } d \Biggr) + H^{ \varepsilon } \cdot \Biggl( \varphi \Biggl( \frac{ \lv ( y , s ) \rv }{ r } \Biggr) \nabla^{ \prime } d \Biggr)\,d \mathcal{ H }^{ n + 1 } ( y , s )\\
				&= - \frac{ 1 }{ \sqrt{ 1 - a^2 } } \int_{ ( \partial^* S^{ \varepsilon } )_i } \frac{ 1 }{ r } P_{ T_{ ( y , s ) } ( \partial^* S^{ \varepsilon } ) } \Biggl( \frac{ ( y , s ) }{ \lv ( y , s ) \rv } \Biggr) : \nabla^{ \prime } \varphi \Biggl( \frac{ \lv ( y , s ) \rv }{ r }  \Biggr) \otimes \nabla^{ \prime } d\\
				&\quad + \varphi \Biggl( \frac{ \lv ( y , s ) \rv }{ r } \Biggr) \mathrm{ div }_{ T_{ ( y , s ) } ( \partial^* S^{ \varepsilon } )_i } ( \nabla^{ \prime } d ) + \varphi \Biggl( \frac{ \lv ( y , s ) \rv }{ r }  \Biggr) H^{ \varepsilon } \cdot \nabla^{ \prime } d \,d \mathcal{ H }^{ n + 1 } ( y , s ),
			\end{split}
		\end{equation*}
		where we used $ \nabla^{ \prime } d = - \nu_{ \partial M \times ( 0 , \infty ) } $ on the boundary. Since $ M $ is a compact smooth bounded domain, we can fix a constant $ c = c ( M , \delta ) $ such that
		\[
		\lv \mathrm{ div }_{ T_{ ( y , s ) } ( \partial^* S^{ \varepsilon } ) } ( \nabla^{ \prime } d ( y , s ) ) \rv \leq c \quad \text{ for all } ( y , s ) \in \{ ( y , s ) \in M \times ( 0 , \infty ) \mid d ( y , s ) \leq \delta \}.
		\]
		By the above inequality, it follows from $ \lv \nabla^{ \prime } d \rv \leq 1 $ and the definition of $ \varphi $ that
		\[
		\sigma_i ( B_{ r / 2 } ( x , z ) ) \leq \frac{ 1 }{ \sqrt{ 1 - a^2 } } \int_{ ( \partial^* S^{ \varepsilon } )_i \cap B_r ( x , z ) } \Biggl( c + \frac{ 3 }{ r } + \lv H^{ \varepsilon } \rv \Biggr)\,d \mathcal{ H }^{ n + 1 }. 
		\]
		Therefore, by considering the covering of $ M \times ( z_1 / 2 , z_2 / 2 ) $ by the balls $ B_{ r / 2 } ( x , z ) $ and \eqref{weak_reg}, we have
		\[
		\lV \delta \mathcal{ H }^{ n + 1 } \llcorner_{ ( \partial^* S^{ \varepsilon } )_i } \rV ( X ) \leq C \sup \lv X \rv \int_{ ( \partial^* S^{ \varepsilon } )_i \cap ( M \times ( z_1 , z_2 ) ) } ( 1 + \lv H^{ \varepsilon } \rv )\,d \mathcal{ H }^{ n + 1 }
		\]
		for all $ X \in C^1_c ( \R^{ n + 1 } \times ( 0 , \infty ) ; \R^{ n + 2 } ) $ supported in $ \R^{ n + 1 } \times ( z_1 , z_2 ) $. Moreover, it follows from the above estimate, \eqref{weak_reg} and Proposition \ref{boundedness_first_variation} that
		\begin{equation*}
			\begin{split}
				\lV \delta \mathcal{ H }^{ n + 1 } \llcorner_{ ( \partial^* S^{ \varepsilon } )_b } \rV ( X ) &\leq \frac{ 1 }{ a } \lV \delta \mathcal{ H }^{ n + 1 } \llcorner_{ ( \partial^* S^{ \varepsilon } )_i } + a \delta \mathcal{ H }^{ n + 1 } \llcorner_{ ( \partial^* S^{ \varepsilon } )_b } \rV ( X ) + \frac{ 1 }{ a }\lV \delta \mathcal{ H }^{ n + 1 } \llcorner_{ ( \partial^* S^{ \varepsilon } )_i } \rV ( X )\\
				&\leq C \sup \lv X \rv \int_{ \partial^* S^{ \varepsilon } \cap ( M \times ( z_1 , z_2 ) ) } ( 1 + \lv H^{ \varepsilon } \rv )\,d \mathcal{ H }^{ n + 1 }
			\end{split}
		\end{equation*}
		for all $ X \in C^1_c ( \R^{ n + 1 } \times ( 0 , \infty ) ; \R^{ n + 2 } ) $ supported in $ \R^{ n + 1 } \times ( z_1 , z_2 ) $. Thus, we obtain \eqref{firstV_esti}.
	\end{proof}
	
	Note that $ \mu_t^{ \varepsilon } $ and $ \nu_t^{ \varepsilon } $ are merely translations of $ \partial^* S^{ \varepsilon } $, which implies that $ ( \mu_t^{ \varepsilon } , \nu_t^{ \varepsilon } ) $ satisfies the assumption \eqref{mass_bound_assumption_for_Brakke_flow2} for every $ t > 0 $ and $ \varepsilon > 0 $. Therefore, thanks to \eqref{themassbddS} and \eqref{firstV_esti}, Theorem \ref{compactness_Brakke} can be applied to $ ( \mu_t^{ \varepsilon } , \nu_t^{ \varepsilon } ) $. Moreover, Ilmanen proved that such a convergent flow is $ z $-invariant, and one can see the same conclusion for the capillary setting (see \cite[Section 8.8]{ilmanen1994elliptic} for details).
	
	\begin{lemma}
		\label{translative_Brakkeflow}
		Taking a subsequence if necessary, there exists the Brakke flow $ ( \overline{ \mu }_t , \overline{ \nu }_t ) $ on $ M \times \R $ with the contact angle $ \theta $ for almost every time $ t > 0 $ such that $ \mu^{ \varepsilon }_t + a \nu^{ \varepsilon }_t \rightharpoonup \overline{ \mu }_t + a \overline{ \nu }_t $ as Radon measures for all $ t > 0 $ and the following holds: Let $ t > 0 $ and $ \phi \in C^2_c ( \R^{ n + 1 } \times \R ; [ 0 , \infty ) ) $ with $ \D \phi \in \mathcal{ T }_{ \partial M \times \R } C^1_c ( \R^{ n + 1 } \times \R ; \R^{ n + 2 } ) $ be arbitrary, we define $ \phi^{ \tau } ( x , z ) = \phi ( x , z - \tau ) $. Then, by Lemma \ref{Basic_props} (3), we have
		\begin{enumerate}
			\item for all $ \tau \geq 0 $, $ \lim_{ s \to t^+ } ( \overline{ \mu }_s + a \overline{ \nu }_s ) ( \phi ) \leq ( \overline{ \mu }_t + a \overline{ \nu }_t ) ( \phi^{ \tau } ) \leq ( \overline{ \mu }_t + a \overline{ \nu }_t ) ( \phi ) $;
			\item for all $ \tau < 0 $, $ ( \overline{ \mu }_t + a \overline{ \nu }_t ) ( \phi ) \leq ( \overline{ \mu }_t + a \overline{ \nu }_t ) ( \phi^{ \tau } ) \leq \lim_{ s \to t^- } ( \overline{ \mu }_s + a \overline{ \nu }_s ) ( \phi ) $;
			\item in particular, we obtain $ ( \overline{ \mu }_t + a \overline{ \nu }_t ) ( \phi^{ \tau } ) = ( \overline{ \mu }_t + a \overline{ \nu }_t ) ( \phi ) $ for all $ \tau \in \R $ and $ t \geq 0 $ except for countable many $ t $.
		\end{enumerate}
		Furthermore, for almost every $ t > 0 $, we have $ \mu^{ \varepsilon }_t  \rightharpoonup \overline{ \mu }_t $, $ \nu^{ \varepsilon }_t  \rightharpoonup \overline{ \nu }_t $ as integral varifolds by passing to another subsequence, and $ \overline{ \nu }_t $ is a unit density ($ n + 1 $)-rectifiable Radon measure by Lemma \ref{W_dens_lemma}. Since $ \overline{ \nu }_t $ is unit density and $ \overline{ \mu }_t $ is integer, $ \overline{ \mu }_t $ and $ \overline{ \nu }_t $ also satisfy the property (3) by considering the densities of $ \overline{ \mu }_t $ and $ \overline{ \nu }_t $ from Proposition \ref{constancy}.
	\end{lemma}
	
	We refer to a Radon measure satisfying (3) as being $ z $-invariant. Ilmanen proved the following product lemma for any $ z $-invariant Radon measure in \cite[Section 8.5]{ilmanen1994elliptic}.
	
	\begin{lemma}
		\label{product_lemma}
		Let $ \overline{ \mu } $ be a Radon measure on $ \R^{ n + 1 } \times \R $. Assume that $ \overline{ \mu } $ is $ z $-invariant, that is, for any $ \phi \in C^0_c ( \R^{ n + 1 } \times \R ) $ and $ \tau \geq 0 $, $ \overline{ \mu } ( \phi^{ \tau } ) = \overline{ \mu } ( \phi ) $, where $ \phi^{ \tau } ( x , z ) = \phi ( x , z - \tau ) $. Then the following hold;
		\begin{enumerate}
			\item if we choose $ \rho \in C^0_c ( \R ; [ 0 , \infty ) ) $ such that $ \int_{ \R } \rho\,d z = 1 $ and define a Radon measure $ \mu $ on $ M $ by $ \mu ( \phi ) := \overline{ \mu } ( \rho \phi ) $ for $ \phi \in C^0_c ( \R^{ n + 1 } ; [ 0 , \infty ) ) $, then $ \mu $ is independent of the choice of $ \rho $ and
			\begin{equation}
				\label{z_invariant}
				\overline{ \mu } = \mu \times \mathcal{ L }^1;
			\end{equation}
			\item if $ \overline{ \mu } \in \RV_{ n + 1 } ( \R^{ n + 1 } \times \R ) $, then $ \mu \in \RV_n ( \R^{ n + 1 } ) $ and $ T_{ ( x , z ) } \overline{ \mu } = T_x \mu \oplus \mathrm{ span } ( \e_z ) $ for $ \mu $-almost every $ x \in \R^{ n + 1 } $ and all $ z \in \R $. If $ \overline{ \mu } \in \IV_{ n + 1 } ( \R^{ n + 1 } \times \R ) $, then $ \mu \in \IV_n ( \R^{ n + 1 } ) $.
		\end{enumerate}
	\end{lemma}
	
	Since the product lemma holds for any $ z $-invariant integral varifold, we can apply them to $ ( \overline{ \mu }_t , \overline{ \nu }_t ) $. Ilmanen also proved the curvature product lemma in \cite[Section 8.5]{ilmanen1994elliptic}, which can be also proved for contact angle varifolds by a simple calculation.
	
	\begin{lemma}
		\label{contantangle_product}
		Let $ ( \overline{ \mu } , \overline{ \nu } ) \in \IV_n ( M \times \R ) \times \IV_n ( \partial M \times \R ) $ be a pair of $ z $-invariant varifolds with the contact angle $ \theta $ in $ M \times \R \subset \R^{ n + 1 } \times \R $ and its mean curvature is denoted by $ H_{ \overline{ \mu } } $. Then we have
		\begin{equation}
			\label{orthogonal_z_MC}
			\quad H_{ \overline{ \mu } } ( x , z ) \cdot \e_z = 0 \text{ for $ \mu $-a.e. } x \in M \text{ and all } z \in \R
		\end{equation}
		and $ ( \mu , \nu ) \in \IV_n ( M ) \times \IV_n ( \partial M ) $ as defined in Lemma \ref{product_lemma} (1) has the contact angle $ \theta $ with the mean curvature $ H_{ \mu } $ in $ M \subset \R^{ n + 1 } $ satisfying $ \mathbf{ p } ( H_{ \overline{ \mu } } ( \cdot , z ) ) = H_{ \mu } ( \cdot ) $ for all $ z \in \R $. Moreover, for any $ \rho \in C^2_c ( \R ; [ 0 , \infty ) ) $ and any $ \phi \in C_c^2 ( \R^{ n + 1 } ; [ 0 , \infty ) ) $ with $ \D \phi \in \mathcal{ T }_{ \partial M } C^1_c ( \R^{ n + 1 } ; \R^{ n + 1 } ) $,
		\begin{equation*}
			\int_{ M \times \R } ( \D' ( \rho \phi ) - \rho \phi H_{ \overline{ \mu } } ) \cdot H_{ \overline{ \mu } }\,d \overline{ \mu } = \int_M ( \D \phi - \phi H_{ \mu } ) \cdot H_{ \mu }\,d \mu.
		\end{equation*}
	\end{lemma}
	\begin{proof}
		Let $ X \in \mathcal{ T }_{ \partial M } C^1_c ( \R^{ n + 1 } ; \R^{ n + 1 } ) $ and $ \rho \in C^2_c ( ( 0 , \infty ) ; [ 0 , \infty ) ) $ with $ \int_0^{ \infty } \rho\,d z = 1 $ fix. From \eqref{contact_angle_condition}, Lemma \ref{product_lemma}, and Fubini's theorem, we obtain
		\begin{equation*}
			\begin{split}
				- \int_{ M \times \R } \phi H_{ \overline{ \mu } } \cdot e_z\,d \overline{ \mu }
				= - \int_{ M \times \R } H_{ \overline{ \mu } } \cdot Y\,d \overline{ \mu }
				&= \int_{ M \times \R } \dv_{ T_{ ( x , z ) } \overline{ \mu } } Y\,d \overline{ \mu } + a \int_{ M \times \R } \dv_{ T_{ ( x , z ) } \overline{ \nu } } Y\,d \overline{ \nu }\\
				&= \int_{ \R } \int_M \partial_z \phi\,d ( \mu + a \nu ) d z = 0
			\end{split}
		\end{equation*}
		for all $ \phi \in C^1_c ( \R^{ n + 1 } \times \R ) $, where we set $ Y = \phi \e_z $, and we used $ \int_{ \R } \partial_z \phi ( x , z )\,d z = 0 $ for all $ x \in M $. This implies that $ H_{ \overline{ \mu } } \cdot \e_z = 0 $ for $ \mu $-almost every $ x \in M $ and all $ z \in \R $. Therefore, thanks to \eqref{contact_angle_condition}, the $ z $-invariant property of $ ( \overline{ \mu }, \overline{ \nu } ) $, Lemma \ref{product_lemma}, and Fubini's theorem, we obtain
		\begin{equation}
			\label{calculas_H_mu}
			\begin{split}
				&\int_M \dv_{ T_x \mu } X\,d \mu + a \int_M \dv_{ T_x \nu } X\,d \nu = \int_{ M \times \R } \rho \dv_{ T_{ ( x , z ) } \overline{ \mu } } \tilde{ X }\,d \overline{ \mu } + a \int_{ M \times \R } \rho \dv_{ T_{ ( x , z ) } \overline{ \nu } } \tilde{ X }\,d \overline{ \nu }\\
				&= \int_{ M \times \R } \dv_{ T_{ ( x , z ) } \overline{ \mu } } ( \rho \tilde{ X } )\,d \overline{ \mu } + a \int_{ M \times \R } \dv_{ T_{ ( x , z ) } \overline{ \nu } } ( \rho \tilde{ X } )\,d \overline{ \nu } - \int_{ M \times \R } \partial_z \rho\,d ( \overline{ \mu } + a \overline{ \nu } )\\
				&= - \int_{ M \times \R } \rho \tilde{ X } \cdot H_{ \overline{ \mu } }\,d \overline{ \mu } - \int_{ \R } \partial_z \rho\,d z \int_M d ( \mu + a \nu ) = - \int_{ M \times \R } \rho \tilde{ X } \cdot H_{ \overline{ \mu } } \,d \overline{ \mu },
			\end{split}
		\end{equation}
		where $ \tilde{ X } = {}^t ( X , 1 ) $ and we used $ \int_{ \R } \partial_z \rho\,d z = 0 $ again. Furthermore, set
		\[
		\tilde{ Y } ( x ) := \int_{ \R } P_{ T_{ ( x , z ) } ( \overline{ \mu } + a \overline{ \nu } ) } ( Y ( x , z ) )\,d z \text{ for } Y \in \mathcal{ T }_{ \partial M \times \R } C^1_c ( \R^{ n + 1 } \times \R ; \R^{ n + 2 } ),
		\]
		we obtain $ \delta ( \overline{ \mu } + a \overline{ \nu } ) ( Y ) = \delta ( \mu + a \nu ) ( \tilde{ Y } ) $ by the definition of the first variation formula. From this and Lemma \ref{product_lemma}, it follows that $ \delta ( \overline{ \mu } + a \overline{ \nu } ) ( Y ) = \delta ( ( \mu + a \nu ) \times \mathcal{ L }^1 ) ( Y ) $ for any $ Y \in \mathcal{ T }_{ \partial M \times \R } C^1_c ( \R^{ n + 1 } \times \R ; \R^{ n + 2 } ) $, which implies that the mean curvature $ H_{ \overline{ \mu } } $ is independent of $ z $. Therefore, by \eqref{calculas_H_mu}, we obtain that $ ( \mu , \nu ) $ has the contact angle $ \theta $ and $ \mathbf{ p } ( H_{ \overline{ \mu } } ( \cdot , z ) ) = H_{ \mu } ( \cdot ) $ for all $ z \in \R $. By \eqref{orthogonal_z_MC} and $ \mathbf{ p } ( H_{ \overline{ \mu } } ) = H_{ \mu } $, it follows from $ \int_{ \R } \rho\,d z = 1 $ that
		\begin{equation*}
			\begin{split}
				\int_{ M \times \R } ( \D' ( \rho \phi ) - \rho \phi H_{ \overline{ \mu } } ) \cdot H_{ \overline{ \mu } }\,d \overline{ \mu } &= \int_{ M \times \R } \phi \D' \rho \cdot H_{ \overline{ \mu } } + \rho ( ( \D' \phi - \phi H_{ \overline{ \mu } } ) \cdot H_{ \overline{ \mu } } )\,d \overline{ \mu }\\
				&=\int_M ( \D \phi - \phi H_{ \mu } ) \cdot H_{ \mu }\,d \mu
			\end{split}
		\end{equation*}
		for any $ \phi \in C^2_c ( \R^{ n + 1 } ; [ 0 , \infty ) ) $ with $ \D \phi \in \mathcal{ T }_{ \partial M } C^1_c ( \R^{ n + 1 } ; \R^{ n + 1 } ) $. This completes the proof.
	\end{proof}
	
	Finally, we obtain the existence theorem of the Brakke flow with contact angle from Lemma \ref{translative_Brakkeflow}, Lemma \ref{product_lemma}, and Lemma \ref{contantangle_product}.
	
	\begin{theorem}
		\label{Brakke_via_Ilmanen}
		Let $ ( \overline{ \mu }_t , \overline{ \nu }_t ) $ be as in Lemma \ref{translative_Brakkeflow}. Let $ \rho \in C^2_c ( ( 0 , \infty ) ; [ 0 , \infty ) ) $ be such that $ \int_0^{ \infty } \rho\,d z = 1 $ and fix. Set $ \mu_t ( \phi ) := \overline{ \mu }_t ( \rho \phi ) $ and $ \nu_t ( \phi ) := \overline{ \nu }_t ( \rho \phi ) $ for $ \phi \in C^0_c ( \R^{ n + 1 } ; [ 0 , \infty ) ) $, then the following hold;
		\begin{enumerate}
			\item for all $ t > 0 $ except for a countable set, $ \overline{ \mu }_t + a \overline{ \nu }_t = ( \mu_t + a \nu_t ) \times \mathcal{ L }^1 $;
			\item $ ( \mu_t , \nu_t ) $ is a Brakke flow and it has the contact angle $ \theta $ in $ M \subset \R^{ n + 1 } $.
		\end{enumerate}
	\end{theorem}
	
	When applying the compactness theorem of Brakke flows, one can take a further subsequence using the compactness theorem of sets of finite perimeter by \eqref{themassbddE}: there exists a set of finite perimeter $ E \subset M \times [ 0 , \infty ) $ such that
	\[
	\chi_{ E^{ \varepsilon } } \to \chi_E \text{ in } L^1_{ loc } ( M \times [ 0 , \infty ) ) , \quad \lV \D' \chi_E \rV \leq \liminf_{ \varepsilon \to 0 } \lV \D' \chi_{ E^{ \varepsilon } } \rV,
	\]
	and we can show that $ ( \mu_t , \nu_t ) $ and $ E_t $ fit as the concept of the enhanced motion in \cite[Section 8.1]{ilmanen1994elliptic}. Furthermore, we ensure that the Brakke flow $ ( \mu_t , \nu_t ) $ starts continuously with the given initial datum $ \partial^* E_0 $.
	
	\begin{prop}
		\label{enhancedmotion}
		For $ ( \mu_t , \nu_t ) $ and $ E_t $ constructed above, we have the following properties;
		\begin{enumerate}
			\item $ \lim_{ t \to 0^+ } ( \mu_t + a \nu_t ) = \mu_0 + a \nu_0 = \lV \D \chi_{ E_0 } \rV \llcorner_{ M^{ \circ } } + a \lV \D \chi_{ E_0 } \rV \llcorner_{ \partial M } $, where we take the limit over the times $ t $ where (2) holds;
			\item for almost every $ t > 0 $, $ \lV \D \chi_{ E_t } \rV \llcorner_{ M^{ \circ } } + a \lV \D \chi_{ E_t } \rV \llcorner_{ \partial M } \leq \mu_t + a \nu_t $;
			\item for all Borel set $ I \subset [ 0 , \infty ) $,
			\[
			\lV \D' \chi_E \rV ( M \times I ) \leq \frac{ 1 }{ a } \Bigl( \mathcal{ L }^1 ( I ) + ( \mathcal{ L }^1 ( I ) )^{ \frac{ 1 }{ 2 } } \Bigr) \mathcal{ H }^n ( \partial^* E_0 ).
			\]
		\end{enumerate}
	\end{prop}
	\begin{proof}
		First of all, by \eqref{themassbddE}, we have
		\begin{equation*}
			\lV \D' \chi_E \rV ( M \times I ) \leq \liminf_{ \varepsilon \to 0 } \lV \D' \chi_{ E^{ \varepsilon } } \rV ( M \times I ) \leq \frac{ 1 }{ a } \Bigl( \mathcal{ L }^1 ( I ) + ( \mathcal{ L }^1 ( I ) )^{ \frac{ 1 }{ 2 } } \Bigr) \mathcal{ H }^n ( \partial^* E_0 ),
		\end{equation*}
		where we used the lower semi-continuous of variation measures. Thus (3) is proved. Let $ z > 0 $ be arbitrary. Since the parallel translation is continuous in $ L^1 $, we have $ \chi_{ E^{ \varepsilon } + \varepsilon z \e_z } \to \chi_E $ in $ L_{ loc }^1 ( \R^{ n + 1 } \times [ 0 , \infty ) ) $. Taking further subsequence from $ \{ E^{ \varepsilon } + \varepsilon z \e_z \} $ if necessary, we see that $ \chi_{ E^{ \varepsilon } } ( x , t + \varepsilon z ) $ converges to $ \chi_{ E }( x , t ) $ as $ \varepsilon \to 0 $ for almost every $ ( x , t ) \in \R^{ n + 1 } \times [ 0 , \infty ) $, and by Fubini's theorem, we have $ \chi_{ E^{ \varepsilon }_{ t + \varepsilon z } } \to \chi_{ E_t } $ in $ L^1 ( \R^{ n + 1 } ) $ for almost ever time $ t > 0 $. Thus we obtain $ \lV \D \chi_{ E_t } \rV \leq \liminf_{ \varepsilon \to 0 } \lV \D \chi_{ E^{ \varepsilon }_{ t + \varepsilon z } } \rV $ for almost every time $ t \geq 0 $ by the lower semi-continuity of variation measures. Note that it follows from the definition of $ S^{ \varepsilon } $ and $ E^{ \varepsilon } $ that $ S^{ \varepsilon } ( t )_z = S^{ \varepsilon }_{ z + t / \varepsilon } = E^{ \varepsilon }_{ t + \varepsilon z } $. We also note that it satisfies
		\begin{equation*}
			\begin{split}
				\lV \D \chi_{ E_t } \rV \llcorner_{ M^{ \circ } } + a \lV \D \chi_{ E_t } \rV \llcorner_{ \partial M } &= a \lV \D \chi_{ E_t } \rV + ( 1 - a ) \lV \D \chi_{ E_t } \rV \llcorner_{ M^{ \circ } }\\
				&\leq \liminf_{ \varepsilon \to 0 } ( a \lV \D \chi_{ E^{ \varepsilon }_{ t + \varepsilon z } } \rV + ( 1 - a ) \lV \D \chi_{ E^{ \varepsilon }_{ t + \varepsilon z } } \rV \llcorner_{ M^{ \circ } } )\\
				&= \liminf_{ \varepsilon \to 0 } ( \lV \D \chi_{ E^{ \varepsilon }_{ t + \varepsilon z } } \rV \llcorner_{ M^{ \circ } } + a \lV \D \chi_{ E^{ \varepsilon }_{ t + \varepsilon z } } \rV \llcorner_{ \partial M } ).
			\end{split} 
		\end{equation*}
		Let $ \phi \in C^0_c ( \R^{ n + 1 } ; [ 0 , \infty ) ) $ and $ \rho \in C^2_c ( \R ; [ 0 , \infty ) ) $ with $ \int_{ \R } \rho\,d z = 1 $ be arbitrary. Then we obtain
		\begin{equation*}
			\begin{split}
				&\mu_t ( \phi ) + a \nu_t ( \phi )
				= \overline{ \mu }_t ( \rho \phi ) + a \overline{ \nu }_t ( \rho \phi )
				= \lim_{ \varepsilon \to 0 } \int_{ M \times ( - z / \varepsilon , \infty ) } \rho \phi\,d ( \mu_t^{ \varepsilon } + a \nu_t^{ \varepsilon } )\\
				&\geq \liminf_{ \varepsilon \to 0 } \int^{ \infty }_{ - z / \varepsilon } \rho \Biggl( \int_{ M^{ \circ } } \phi\,d \lV \D \chi_{ S^{ \varepsilon } ( t )_z } \rV + a \int_{ \partial M } \phi\,d \lV \D \chi_{ S^{ \varepsilon } ( t )_z } \rV \Biggr)\,d z\\
				&= \liminf_{ \varepsilon \to 0 } \int^{ \infty }_{ - z / \varepsilon } \rho \Biggl( a \int_M \phi\,d \lV \D \chi_{ S^{ \varepsilon } ( t )_z } \rV + ( 1 - a ) \int_{ M^{ \circ } } \phi\,d \lV \D \chi_{ S^{ \varepsilon } ( t )_z } \rV \Biggr)\,d z\\
				&\geq a \int_{ \R } \rho \liminf_{ \varepsilon \to 0 } \lV \D \chi_{ E^{ \varepsilon }_{ t + \varepsilon z } } \rV ( \phi )\,d z + ( 1 - a ) \int_{ \R } \rho \liminf_{ \varepsilon \to 0 } \lV \D \chi_{ E^{ \varepsilon }_{ t + \varepsilon z } } \rV \llcorner_{ M^{ \circ } } ( \phi )\,d z\\
				&\geq a \lV \D \chi_{ E_t } \rV ( \phi ) + ( 1 - a ) \lV \D \chi_{ E_t } \rV \llcorner_{ M^{ \circ } } ( \phi )
				= \lV \D \chi_{ E_t } \rV \llcorner_{ M^{ \circ } } ( \phi ) + a \lV \D \chi_{ E_t } \rV \llcorner_{ \partial M } ( \phi ),
			\end{split}
		\end{equation*}
		where we used the co-area formula, Fatou's Lemma, the lower semi-continuity of $ \lV \D \chi_{ E_t } \rV $, and $ \int_{ \R } \rho\,d z = 1 $. Next, by taking the limit in \eqref{themassbddS}, we have
		\begin{equation*}
			\begin{split}
				( \overline{ \mu }_t + a \overline{ \nu }_t ) ( M \times I ) \leq \mathcal{ L }^1 ( I ) ( \mathcal{ H }^n ( ( \partial^* E_0 )_i ) + a \mathcal{ H }^n ( ( \partial^* E_0 )_b ) )
			\end{split}
		\end{equation*}
		for all $ t \geq 0 $. Here let $ \rho \in C^2_c ( \R ; [ 0 , \infty ) ) $ be such that $ \int_{ \R } \rho\,d z = 1 $. Then, by uniform approximation of $ \rho $ by the step functions, we have
		\begin{equation}
			\label{enhanced_1}
			( \mu_t + a \nu_t ) ( M ) = \int_{ M \times \R } \rho\,d ( \overline{ \mu }_t + a \overline{ \nu }_t ) \leq  \mathcal{ H }^n ( ( \partial^* E_0 )_i ) + a \mathcal{ H }^n ( ( \partial^* E_0 )_b )
		\end{equation}
		for all $ t \geq 0 $. On the other hand, use the lower semi-continuous of variation measures, (2) and Proposition \ref{Basic_props} (2), then we have
		\begin{equation}
			\label{enhanced_2}
			\begin{split}
				\lV \D \chi_{ E_0 } \rV \llcorner_{ M^{ \circ } } + a \lV \D \chi_{ E_0 } \rV \llcorner_{ \partial M } &\leq \liminf_{ t \to 0^+ } ( \lV \D \chi_{ E_t } \rV \llcorner_{ M^{ \circ } } + a \lV \D \chi_{ E_t } \rV \llcorner_{ \partial M } )\\
				&\leq \lim_{ t \to 0^+ } ( \mu_t + a \nu_t ) \leq \mu_0 + a \nu_0.
			\end{split}
		\end{equation}
		Here, we take the limit over the times $ t $ where (2) holds in \eqref{enhanced_2}. Thus (1) follows from \eqref{enhanced_1} and \eqref{enhanced_2}.
	\end{proof}
	
	\subsection{Continuity property of $ E_t $}
	Finally, we discuss the continuity of the volume change of $ E_t $. To this end, we prove that $ S^{ \varepsilon } $ extends in the $ z $ direction in the measure-theoretic sense.
	
	\begin{lemma}
		\label{time_disconti_0}
		For $ S^{ \varepsilon } $, we have
		\[
		\lim_{ \varepsilon \to 0 } \lV \D' \chi_{ S^{ \varepsilon } } \rV \bigl( \bigl\{ ( x , z ) \in \partial^* S^{ \varepsilon } : \lv P_{ T_{ ( x , z ) } ( \partial^* S^{ \varepsilon } ) } ( \e_z ) \rv^2 \leq 1 / 2 \bigr\} \bigr) = 0.
		\]
	\end{lemma}
	\begin{proof}
		Since $\lv P_{ T_{ ( x , z ) } ( \partial^* S^{ \varepsilon } ) } ( \e_z ) \rv = 1 $ on $ \partial M \times ( 0 , \infty ) $, we have
		\[
		\lV \D' \chi_{ S^{ \varepsilon } } \rV \bigl( \bigl\{ \lv P_{ T_{ ( x , z ) } ( \partial^* S^{ \varepsilon } ) } ( \e_z ) \rv^2 \leq 1 / 2 \bigr\} \bigr) = \lV \D' \chi_{ S^{ \varepsilon } } \rV \llcorner_{ M^{ \circ } \times ( 0 , \infty ) } \bigl( \bigl\{ \lv P_{ T_{ ( x , z ) } ( \partial^* S^{ \varepsilon } ) } ( \e_z ) \rv^2 \leq 1 / 2 \bigr\} \bigr).
		\]
		From Lemma \ref{E_L_eq} (1) and $ 1 = \lv P_{ T_{ ( x , z ) } ( \partial^* S^{ \varepsilon } ) } ( \e_z ) \rv^2 + \lv P_{ T_{ ( x , z ) } ( \partial^* S^{ \varepsilon } ) }^{ \perp } ( \e_z ) \rv^2 $, we obtain
		\[
		\lV \D' \chi_{ S^{ \varepsilon } } \rV \llcorner_{ M^{ \circ } \times ( 0 , \infty ) } \bigl( \bigl\{ \lv P_{ T_{ ( x , z ) } ( \partial^* S^{ \varepsilon } ) } ( \e_z ) \rv^2 \leq 1 / 2 \bigr\} \bigr) = \lV \D' \chi_{ S^{ \varepsilon } } \rV \llcorner_{ M^{ \circ } \times ( 0 , \infty ) } \bigl( \bigl\{ 1 \leq 2 \varepsilon^2 \lv H^{ \varepsilon } \rv^2 \bigr\} \bigr).
		\]
		Thus, by Markov's inequality and \eqref{L2MC_estimate}, we compute
		\begin{equation*}
			\begin{split}
				&\lV \D' \chi_{ S^{ \varepsilon } } \rV \bigl( \bigl\{ \lv P_{ T_{ ( x , z ) } ( \partial^* S^{ \varepsilon } ) } ( \e_z ) \rv^2 \leq 1 / 2 \bigr\} \bigr) = \lV \D' \chi_{ S^{ \varepsilon } } \rV \llcorner_{ M^{ \circ } \times ( 0 , \infty ) } \bigl( \bigl\{ 1 \leq 2 \varepsilon^2 \lv H^{ \varepsilon } \rv^2 \bigr\} \bigr)\\
				&\leq 2 \varepsilon \int_{ ( \partial^* S^{ \varepsilon } )_i } \varepsilon \lv H^{ \varepsilon } \rv^2\,d \mathcal{ H }^{ n + 1 } \leq 2 \varepsilon \mathcal{ H }^n ( \partial^* E_0 ).
			\end{split}
		\end{equation*}
		Letting $ \varepsilon \to 0 $, we have the conclusion.
	\end{proof}
	
	Using the above lemma to apply the co-area formula to $ S^{ \varepsilon } $, we obtain the following.
	
	\begin{prop}
		\label{Holderconti_E_t}
		The characteristic function $ \chi_{ E_t } $ is $ 1 / 2 $-H\"{o}lder continuous in $ L^1 $ with respect to $ t \geq 0 $ except for a null set. 
	\end{prop}
	\begin{proof}
		We set
		\[
		\Sigma^{ \varepsilon } := \bigl\{ ( x , z ) \in \partial^* S^{ \varepsilon } : \lv P_{ T_{ ( x , z ) } ( \partial^* S^{ \varepsilon } ) } ( \e_z ) \rv^2 \leq 1 / 2 \bigr\}
		\]
		and let $ \phi \in C^1_c ( \R^{ n + 1 } \times ( 0 , \infty ) ) $ be arbitrary. We define the approximate velocity of $ E^{ \varepsilon } $ by
		\[
		V^{ \varepsilon } ( x , t ) :=
		\begin{cases}
			- \frac{ \mathbf{ q } ( \nu_{ E^{ \varepsilon } } ) }{ \lv \mathbf{ p } ( \nu_{ E^{ \varepsilon } } ) \rv } ( x , t ) &( ( x , t ) \in \kappa_{ \varepsilon } ( \partial^* S^{ \varepsilon } \setminus \Sigma^{ \varepsilon } ) )\\
			0 &( ( x , t ) \in \kappa_{ \varepsilon } ( \Sigma^{ \varepsilon } ) ).
		\end{cases}
		\]
		Since the map $ \kappa_{ \varepsilon } $ shrinks $ z $-variable by $ \varepsilon $, the following hold:
		\begin{align}
			\frac{ \mathbf{ q } ( \nu_{ S^{ \varepsilon } } ) }{ \lv \mathbf{ p } ( \nu_{ S^{ \varepsilon } } ) \rv } ( x , z ) &= \varepsilon \frac{ \mathbf{ q } ( \nu_{ E^{ \varepsilon } } ) }{ \lv \mathbf{ p } ( \nu_{ E^{ \varepsilon } } ) \rv } ( x , t ), \quad t = \varepsilon z, \label{relationship_nu_S_E}\\
			\int_{ E^{ \varepsilon } } \pt \phi\,d x d t &= \int_{ S^{ \varepsilon } } \partial_z \phi\,d x d z. \label{mass_t_derivatibve_S_E}
		\end{align}
		Thus, by \eqref{relationship_nu_S_E} and \eqref{mass_t_derivatibve_S_E}, we calculate
		\begin{equation}
			\label{calc_ptE}
			\begin{split}
				&\int_{ E^{ \varepsilon } } \pt \phi\,d x d t = \int_{ S^{ \varepsilon } } \partial_z \phi\,d x d z = \int_{ ( M \times ( 0 , \infty ) ) \cap \partial^* S^{ \varepsilon } } ( \chi_{ \Sigma^{ \varepsilon } } + \chi_{ \partial^* S^{ \varepsilon } \setminus \Sigma^{ \varepsilon } } ) \phi \mathbf{ q } ( \nu_{ S^{ \varepsilon } } )\,d \mathcal{ H }^{ n + 1 }\\
				&= \int_{ ( M \times ( 0 , \infty ) ) \cap \partial^* S^{ \varepsilon } } \chi_{ \Sigma^{ \varepsilon } } \phi \mathbf{ q } ( \nu_{ S^{ \varepsilon } } )\,d \mathcal{ H }^{ n + 1 } + \int_0^{ \infty } \int_{ M \cap \partial^* S^{ \varepsilon }_z } \chi_{ \partial^* S^{ \varepsilon } \setminus \Sigma^{ \varepsilon } } \phi \frac{ \mathbf{ q } ( \nu_{ S^{ \varepsilon } } ) }{ \lv \mathbf{ p } ( \nu_{ S^{ \varepsilon } } ) \rv } \,d \mathcal{ H }^n d z\\
				&= \int_{ ( \partial^* S^{ \varepsilon } )_i } \chi_{ \Sigma^{ \varepsilon } } \phi \mathbf{ q } ( \nu_{ S^{ \varepsilon } } )\,d \mathcal{ H }^{ n + 1 } - \int_0^{ \infty } \int_{ ( \partial^* E^{ \varepsilon }_t )_i } \phi V^{ \varepsilon }\,d \mathcal{ H }^n d t,
			\end{split}
		\end{equation}
		where we used the co-area formula and $ \mathbf{ q } ( \nu_{ S^{ \varepsilon } } ) = 0 $ on $ \partial M $. From Lemma \ref{time_disconti_0}, we have
		\[
		\lim_{ \varepsilon \to 0 } \int_{ ( \partial^* S^{ \varepsilon } )_i } \chi_{ \Sigma^{ \varepsilon } } \phi \mathbf{ q } ( \nu_{ S^{ \varepsilon } } )\,d \mathcal{ H }^{ n + 1 } = 0.
		\]
		It follows from a simple geometric argument that $ \lv \mathbf{ p } ( \nu_{ S^{ \varepsilon } } ) \rv = \lv P_{ T_{ ( x , z ) } ( \partial^* S^{ \varepsilon } ) } ( \e_z ) \rv $, and by the definition of $ \Sigma^{ \varepsilon } $, we see that $ 1 / 2 < \lv \mathbf{ p } ( \nu_{ S^{ \varepsilon } } ) \rv^2 $ on $ \partial^* S^{ \varepsilon } \setminus \Sigma^{ \varepsilon } $. Hence, by Lemma \ref{E_L_eq} (1), we see that
		\[
		\Biggl( \chi_{ \partial^* S^{ \varepsilon } \setminus \Sigma^{ \varepsilon } } \frac{ \mathbf{ q } ( \nu_{ S^{ \varepsilon } } ) }{ \lv \mathbf{ p } ( \nu_{ S^{ \varepsilon } } ) \rv } \Biggr)^2 \leq 2 \varepsilon^2 \lv H^{ \varepsilon } \rv^2 \text{ in } M^{ \circ }.
		\]
		Therefore, by \eqref{L2MC_estimate}, we compute
		\begin{equation*}
			\begin{split}
				&\int_0^{ \infty } \int_{ ( \partial^* E^{ \varepsilon }_t )_i } \lv V^{ \varepsilon } \rv^2\,d \mathcal{ H }^n d t = \int_0^{ \infty } \int_{ ( \partial^* S^{ \varepsilon }_t )_i } \frac{ 1 }{ \varepsilon } \Biggl( \chi_{ \partial^* S^{ \varepsilon } \setminus \Sigma^{ \varepsilon } } \frac{ \mathbf{ q } ( \nu_{ S^{ \varepsilon } } ) }{ \lv \mathbf{ p } ( \nu_{ S^{ \varepsilon } } ) \rv } \Biggr)^2\,d \mathcal{ H }^n d z\\
				&\leq 2 \int_0^{ \infty } \int_{ ( \partial^* S^{ \varepsilon }_t )_i } \varepsilon \lv H^{ \varepsilon } \rv^2\,d \mathcal{ H }^n d z \leq 2 \int_{ ( \partial^* S^{ \varepsilon } )_i } \varepsilon \lv H^{ \varepsilon } \rv^2\,d \mathcal{ H }^{ n + 1 } \leq 2 \mathcal{ H }^n ( \partial^* E_0 ),
			\end{split}
		\end{equation*}
		where we used the co-area formula. Thus, taking the limit $ \varepsilon \to 0 $ in \eqref{calc_ptE}, we obtain
		\begin{equation}
			\label{V_estimate}
			\Biggl\lv \int_E \pt \phi\, d x d t \Biggr\rv \leq C \lim_{ \varepsilon \to 0 } \Biggl( \int_0^{ \infty } \int_{ ( \partial^* E^{ \varepsilon }_t )_i } \lv \phi \rv^2\,d\mathcal{ H }^n d t \Biggr)^{ \frac{ 1 }{ 2 } }
		\end{equation}
		for some constant $ C > 0 $ depending only on $ E_0 $.
		\vskip.3\baselineskip
		Let $ G \subset [ 0 , \infty ) $ be a set of Lebesgue points of $ f ( t ) := \int_{ E_t } d x $. It follows from a standard result in measure theory that the set $ G $ is full measure in $ [ 0 , \infty ) $. Let $ t_1 < t_2 $ be arbitrary points in $ G $. We choose $ \phi $ as depending only on time, and by approximating $ \phi \to \chi_{ [ t_1 , t_2 ] } $, we obtain the following from \eqref{themassbddE} and \eqref{V_estimate}:
		\begin{equation*}
			\lv \mathcal{ L }^{ n + 1 } ( E_{ t_2 } ) - \mathcal{ L }^{ n + 1 } ( E_{ t_1 } ) \rv \leq C \mathcal{ H }^n ( \partial^* E_0 )^{ \frac{ 1 }{ 2 } } ( t_2 - t_1 )^{ \frac{ 1 }{ 2 } },
		\end{equation*}
		where we used the co-area formula. This completes the proof.
	\end{proof}
	\begin{remark}
		By Lemma \ref{Holderconti_E_t}, one may re-define the set $ E $ so that $ \chi_{ E_t } $ is $ 1 / 2 $-H\"{o}lder continuous in $ L^1 $ to the time direction for all $ t > 0 $. We also note that, by approaching from the left, one can re-define the Brakke flow $ ( \mu_t , \nu_t ) $ so that $ \mu_t + a \nu_t $ is left-continuous for all $ t \geq 0 $ from Lemma \ref{Basic_props} while keeping the Brakke inequality \eqref{Brakkeineq}. Furthermore, we can prove that the property $ \lV \D \chi_{ E_t } \rV \llcorner_{ M^{ \circ } } + a \lV \D \chi_{ E_t } \rV \llcorner_{ \partial M } \leq \mu_t + a \nu_t $ also holds for any $ t \geq 0 $. Let $ G $ be the full measure set of $ [ 0 , \infty ) $ where Proposition \ref{Basic_props} (2) holds. For any $ t \notin G $, we can choose a sequence $ \{ t_i \}_i \subset G $ approaching from the left to $ t $. Since $ \chi_{ E_{ t_i } } \to \chi_{ E_t } $ in $ L^1 $, we have, for any $ \phi \in C^0_c ( \R^{ n + 1 } ; [ 0 , \infty ) ) $,
		\begin{equation*}
			\begin{split}
				&( \lV \D \chi_{ E_t } \rV \llcorner_{ M^{ \circ } } + a \lV \D \chi_{ E_t } \rV \llcorner_{ \partial M } ) ( \phi ) \leq \liminf_{ i \to \infty } ( \lV \D \chi_{ E_{ t_i } } \rV \llcorner_{ M^{ \circ } } + a \lV \D \chi_{ E_{ t_i } } \rV \llcorner_{ \partial M } ) ( \phi )\\
				&\leq \liminf_{ i \to \infty } ( \mu_{ t_i } + a \nu_{ t_i } ) ( \phi ) = ( \mu_t + a \nu_t ) ( \phi ),
			\end{split}
		\end{equation*}
		where we used the lower semi-continuous property of variation measures, Proposition \ref{enhancedmotion} (2), and the left-continuity of $ \mu_t + a \nu_t $. This completes all of the claims in Theorem \ref{main_results}.
	\end{remark}
	
	\bibliography{myref.bib}

\begin{thebibliography}{10}

\bibitem{barles1999nonlinear}
G.~Barles.
\newblock Nonlinear neumann boundary conditions for quasilinear degenerate
  elliptic equations and applications.
\newblock {\em J. Differ. Equations}, 154(1):191--224, (1999).

\bibitem{bellettini2018minimizing}
G.~Bellettini and S.~Y. Kholmatov.
\newblock Minimizing movements for mean curvature flow of droplets with
  prescribed contact angle.
\newblock {\em J. Math. Pures Appl.}, 117:1--58, (2018).

\bibitem{bevilacqua2025classical}
G.~Bevilacqua, S.~Stuvard, and B.~Velichkov.
\newblock Classical solutions to the soap film capillarity problem for plane
  boundaries.
\newblock {\em Math. Ann.}, pages 1--53, (2025).

\bibitem{brakke1978motion}
K.~A. Brakke.
\newblock {\em The motion of a surface by its mean curvature}, volume~20 of
  {\em Math. Notes (Princeton)}.
\newblock Princeton University Press, Princeton, 1978.

\bibitem{de2021rectifiability}
L.~De~Masi.
\newblock Rectifiability of the free boundary for varifolds.
\newblock {\em Indiana Univ. Math. J.}, 70(6):2603--2651, (2021).

\bibitem{de2022existence}
L.~De~Masi.
\newblock {\em Existence and properties of minimal surfaces and varifolds with
  contact angle conditions}.
\newblock SISSA, 2022.
\newblock https://hdl.handle.net/20.500.11767/129590.

\bibitem{de2025min}
L.~De~Masi and G.~De~Philippis.
\newblock Min-max construction of minimal surfaces with a fixed angle at the
  boundary.
\newblock {\em J. Differ. Geom.}, 131(2):415--496, (2025).

\bibitem{de2025regularity}
L.~De~Masi, N.~Edelen, C.~Gasparetto, and C.~Li.
\newblock Regularity of minimal surfaces with capillary boundary conditions.
\newblock {\em Commun. Pure Appl. Math.}, 78(12):2436--2502, (2025).

\bibitem{de2024regularity}
G.~De~Philippis, N.~Fusco, and M.~Morini.
\newblock Regularity of capillarity droplets with obstacle.
\newblock {\em Trans. Am. Math. Soc.}, 377(8):5787--5835, (2024).

\bibitem{philippis2015regularity}
G.~De~Philippis and F.~Maggi.
\newblock Regularity of free boundaries in anisotropic capillarity problems and
  the validity of young's law.
\newblock {\em Arch. Ration. Mech. Anal.}, 216(2):473--568, (2015).

\bibitem{Edelen+2020+95+137}
N.~Edelen.
\newblock The free-boundary brakke flow.
\newblock {\em J. Reine Angew. Math. (Crelles Journal)}, 2020(758):95--137,
  (2020).

\bibitem{eto2024minimizing}
T.~Eto and Y.~Giga.
\newblock On a minimizing movement scheme for mean curvature flow with
  prescribed contact angle in a curved domain and its computation.
\newblock {\em Ann. Mat. Pura Appl. (1923-)}, 203(3):1195--1221, (2024).

\bibitem{eto2024convergence}
T.~Eto and Y.~Giga.
\newblock A convergence result for a minimizing movement scheme for mean
  curvature flow with prescribed contact angle in a curved domain.
\newblock {\em Adv. Math. Sci. Appl.}, 34(1):115--147, (2025).

\bibitem{evans2015measure}
L.~C. Evans and R.~F. Gariepy.
\newblock {\em Measure theory and fine properties of functions}.
\newblock Textb. Math. CRC Press, Boca Raton, FL, revised edition, 2015.

\bibitem{gilbarg2001elliptic}
D.~Gilbarg and N.~S. Trudinger.
\newblock {\em Elliptic Partial Differential Equations of Second Order}, volume
  224.
\newblock Springer Science \& Business Media, 2001.

\bibitem{gruter1986allard}
M.~Gr{\"u}ter and J.~Jost.
\newblock Allard type regularity results for varifolds with free boundaries.
\newblock {\em Ann. Sc. Norm. Super. Pisa, Cl. Sci.}, 13(1):129--169, (1986).

\bibitem{hensel2021existence}
S.~Hensel and T.~Laux.
\newblock {BV} solutions for mean curvature flow with constant contact angle:
  {A}llen--{C}ahn approximation and weak-strong uniqueness.
\newblock {\em Indiana Univ. Math. J.}, 73(1):111--148, (2024).

\bibitem{hutchinson1986second}
J.~E. Hutchinson.
\newblock Second fundamental form for varifolds and the existence of surfaces
  minimising curvature.
\newblock {\em Indiana Univ. Math. J.}, 35(1):45--71, (1986).

\bibitem{ilmanen1994elliptic}
T.~Ilmanen.
\newblock Elliptic regularization and partial regularity for motion by mean
  curvature.
\newblock {\em Mem. Am. Math. Soc.}, 108(520), (1994).

\bibitem{ishii2004nonlinear}
H.~Ishii and M.-H. Sato.
\newblock Nonlinear oblique derivative problems for singular degenerate
  parabolic equations on a general domain.
\newblock {\em Nonlinear Anal., Theory Methods Appl.}, 57(7-8):1077--1098,
  (2004).

\bibitem{kagaya2017contactangle}
T.~Kagaya and Y.~Tonegawa.
\newblock A fixed contact angle condition for varifolds.
\newblock {\em Hiroshima Math. J.}, 47(2):139--153, (2017).

\bibitem{kagaya2018singular}
T.~Kagaya and Y.~Tonegawa.
\newblock A singular perturbation limit of diffused interface energy with a
  fixed contact angle condition.
\newblock {\em Indiana Univ. Math. J.}, 67(4):1425--1437, (2018).

\bibitem{kholmatov2024consistency}
S.~Kholmatov.
\newblock Consistency of minimizing movements with smooth mean curvature flow
  of droplets with prescribed contact-angle in $\mathbb{R}^3$.
\newblock {\em Calc. Var. Partial Differ Equ.}, 64(7):211, (2025).

\bibitem{kholmatov2024minimizing}
S.~Kholmatov.
\newblock Minimizing movements for forced anisotropic curvature flow of
  droplets.
\newblock {\em Interfaces Free Bound.}, 27(3):349--402, (2025).

\bibitem{king2022smoothness}
D.~King, F.~Maggi, and S.~Stuvard.
\newblock Smoothness of collapsed regions in a capillarity model for soap
  films.
\newblock {\em Arch. Ration. Mech. Anal.}, 243(2):459--500, (2022).

\bibitem{king2022plateau}
D.~King, S.~Stuvard, and F.~Maggi.
\newblock Plateau's problem as a singular limit of capillarity problems.
\newblock {\em Commun. Pure Appl. Math.}, 75(5):895--969, (2022).

\bibitem{luckhaus1995implicit}
S.~Luckhaus and T.~Sturzenhecker.
\newblock Implicit time discretization for the mean curvature flow equation.
\newblock {\em Calc. Var. Partial Differ. Equ.}, 3(2):253--271, (1995).

\bibitem{maggi2012sets}
F.~Maggi.
\newblock {\em Sets of finite perimeter and geometric variational problems: an
  introduction to Geometric Measure Theory}, volume 135 of {\em Camb. Stud.
  Adv. Math.}
\newblock Cambridge University Press, Cambridge, 2012.

\bibitem{marshallstevens2024gradientflowphasetransitions}
K.~Marshall-Stevens, M.~Takada, Y.~Tonegawa, and M.~Workman.
\newblock Gradient flow of phase transitions with fixed contact angle.
\newblock (2024).
\newblock to appear in Interfaces Free Bound.

\bibitem{SchulzeWhite+2020+281+305}
F.~Schulze and B.~White.
\newblock A local regularity theorem for mean curvature flow with triple edges.
\newblock {\em J. Reine Angew. Math. (Crelles Journal)}, 2020(758):281--305,
  (2020).

\bibitem{simon1983lectures}
L.~Simon.
\newblock {\em Lectures on geometric measure theory}, volume~3 of {\em Proc.
  Cent. Math. Anal. Aust. Natl. Univ.}
\newblock Australian National University, Canberra, 1983.

\bibitem{takasao2016existence}
K.~Takasao and Y.~Tonegawa.
\newblock Existence and regularity of mean curvature flow with transport term
  in higher dimensions.
\newblock {\em Math. Ann.}, 364(3):857--935, (2016).

\bibitem{tashiro2023existence}
K.~Tashiro.
\newblock Existence of {BV} flow via elliptic regularization.
\newblock {\em Hiroshima Math. J.}, 54(2):233--259, (2024).

\bibitem{tonegawa2019brakke}
Y.~Tonegawa.
\newblock {\em {B}rakke's {M}ean {C}urvature {F}low: {A}n {I}ntroduction}.
\newblock SpringerBriefs Math. Springer Singapore, Singapore, 2019.

\bibitem{wang2024allard}
G.~Wang.
\newblock Allard-type regularity for varifolds with prescribed contact angle.
\newblock {\em arXiv preprint arXiv:2403.17415}, (2024).

\bibitem{wang2024monotonicity}
G.~Wang, C.~Xia, and X.~Zhang.
\newblock Monotonicity formulas for capillary surfaces.
\newblock {\em J. Math. Pures Appl.}, 204(103802):pp.30, (2025).

\bibitem{brian2005regularity}
B.~White.
\newblock A local regularity theorem for mean curvature flow.
\newblock {\em Ann. Math.}, 161(3):1487--1519, (2005).

\bibitem{white2010maximum}
B.~White.
\newblock The maximum principle for minimal varieties of arbitrary codimension.
\newblock {\em Commun. Anal. Geom.}, 18(3):421--432, (2010).

\bibitem{WhiteMCF2021}
B.~White.
\newblock Mean curvature flow with boundary.
\newblock {\em Ars Inven. Anal.}, page~43, (2021).
\newblock arXiv:1901.03008.

\end{thebibliography}
	\bibliographystyle{abbrv}
	
\end{document}